\documentclass[12pt]{article}

\usepackage{amsfonts,amssymb,amsmath,amsthm,epsfig,euscript}

\usepackage{tikz}
\usetikzlibrary{matrix,trees,arrows}

\usetikzlibrary{positioning}
\usetikzlibrary{fit}
\usetikzlibrary{patterns}

\newtheorem{theorem}{Theorem}
\newtheorem{lemma}[theorem]{Lemma}

\newtheorem{proposition}{Proposition}

\theoremstyle{definition}
{

}

\long\def\symbolfootnote[#1]#2{\begingroup
\def\thefootnote{\fnsymbol{footnote}}\footnote[#1]{#2}\endgroup}

\newcommand{\red}[1][\sigma]{\mathrm{red}(#1)}

\newcommand{\sg}{\sigma}

\newcommand{\mmp}{\mathrm{mmp}}

\def\A{\mathcal{A}}

\newcommand{\fig}[2]{\begin{figure}[ht]
\centerline{\scalebox{.66}{\epsfig{file=#1.eps}}}
\caption{#2}
\label{fig:#1}
\end{figure}}

\newcommand{\shadetheboxes}[1]{
	\foreach \x/\y in {#1}
      	\fill[pattern color = black!65, pattern=north east lines] (\x,\y) rectangle +(1,1);
	}
	
\newcommand{\drawthegrid}[1]{
	\draw (0.01,0.01) grid (#1+0.99,#1+0.99);
	}
	
\newcommand{\drawtheclpattern}[1]{
	\foreach \x/\y in {#1}
      	\filldraw (\x,\y) circle (6pt);
	}

\newcommand{\drawspecialbox}[1]{
	\foreach \x/\y/\z/\w/\A in {#1}
		{
       		\fill[color = white!100, opacity=1, rounded corners = 1.5pt] (\x+0.125,\y+0.125) rectangle (\z-0.125,\w-0.125);
       		\draw[color = black, rounded corners = 1.5pt] (\x+0.125,\y+0.125) rectangle (\z-0.125,\w-0.125);
       		\fill[black] (\x/2+\z/2,\y/2+\w/2) node {$\scriptstyle\A$};
       	}
    }

\newcommand{\mmpattern}[5]{									
  \raisebox{0.6ex}{
  \begin{tikzpicture}[scale=0.35, baseline=(current bounding box.center), #1]
  \useasboundingbox (0.0,-0.1) rectangle (#2+1.4,#2+1.1);
    
    \shadetheboxes{#4}
    
    \drawthegrid{#2}
    
    \drawspecialbox{#5}
    
    \drawtheclpattern{#3}

  \end{tikzpicture}}
}

\title{Quadrant marked mesh patterns in alternating permutations II}

\author{
Sergey Kitaev \\
\small University of Strathclyde\\[-0.8ex]
\small Livingstone Tower, 26 Richmond Street\\[-0.8ex]
\small Glasgow G1 1XH, United Kingdom\\[-0.8ex]
\small \texttt{sergey.kitaev@cis.strath.ac.uk}
\and
Jeffrey Remmel \\
\small Department of Mathematics\\[-0.8ex]
\small University of California, San Diego\\[-0.8ex]
\small La Jolla, CA 92093-0112. USA\\[-0.8ex]
\small \texttt{jremmel@ucsd.edu}
}

\date{\small Submitted: Date 1;  Accepted: Date 2;
 Published: Date 3.\\
\small MR Subject Classifications: 05A15, 05E05}

\begin{document}
\maketitle

\begin{abstract}
\noindent \

This paper is continuation of the systematic study of distribution of 
quadrant marked mesh patterns initiated in \cite{kitrem}. 
We study quadrant marked mesh patterns on 
up-down and down-up permutations.  \\

\end{abstract}

\section{Introduction}

The notion of mesh patterns was introduced by Br\"and\'en and Claesson \cite{BrCl} to provide explicit expansions for certain permutation statistics as, possibly infinite, linear combinations of (classical) permutation patterns (see \cite{kit} for a comprehensive introduction to the theory of permutation patterns).  This notion was further studied in \cite{AKV,HilJonSigVid,kitrem,kitrem2,kitremtie1,kitremtie2,Ulf}. 

 Let $\sigma = \sg_1 \ldots \sg_n$ be a permutation in 
the symmetric group $S_n$ 
written in one-line notation. Then we will consider the 
graph of $\sg$, $G(\sg)$, to be the set of points $(i,\sg_i)$ for 
$i =1, \ldots, n$.  For example, the graph of the permutation 
$\sg = 471569283$ is pictured in Figure 
\ref{fig:basic}.  Then if we draw a coordinate system centered at a 
point $(i,\sg_i)$, we will be interested in  the points that 
lie in the four quadrants I, II, III, and IV of that 
coordinate system as pictured 
in Figure \ref{fig:basic}.  For any $a,b,c,d \in  
\mathbb{N}$ where $\mathbb{N} = \{0,1,2, \ldots \}$ is the set of 
natural numbers and any $\sg = \sg_1 \ldots \sg_n \in S_n$, 
we say that $\sg_i$ matches the 
quadrant marked mesh pattern $MMP(a,b,c,d)$ in $\sg$ if in $G(\sg)$  relative 
to the coordinate system which has the point $(i,\sg_i)$ as its  
origin,  there are 
$\geq a$ points in quadrant I, 
$\geq b$ points in quadrant II, $\geq c$ points in quadrant 
III, and $\geq d$ points in quadrant IV.  
For example, 
if $\sg = 471569283$, the point $\sg_4 =5$  matches 
the quadrant marked mesh pattern $MMP(2,1,2,1)$ since relative 
to the coordinate system with origin $(4,5)$,  
there are 3 points in $G(\sg)$ in quadrant I, 
1 point in $G(\sg)$ in quadrant II, 2 points in $G(\sg)$ in quadrant III, and 2 points in $G(\sg)$ in 
quadrant IV.  Note that if a coordinate 
in $MMP(a,b,c,d)$ is 0, then there is no condition imposed 
on the points in the corresponding quadrant. In addition, we shall 
consider patterns  $MMP(a,b,c,d)$ where 
$a,b,c,d \in \mathbb{N} \cup \{\emptyset\}$. Here when 
one of the parameters $a$, $b$, $c$, or $d$ in  
$MMP(a,b,c,d)$ is the empty set, then for $\sg_i$ to match  
$MMP(a,b,c,d)$ in $\sg = \sg_1 \ldots \sg_n \in S_n$, 
it must be the case that there are no points in $G(\sg)$ relative 
to coordinate system with origin $(i,\sg_i)$ in the corresponding 
quadrant. For example, if $\sg = 471569283$, the point 
$\sg_3 =1$ matches 
the marked mesh pattern $MMP(4,2,\emptyset,\emptyset)$ since relative 
to the coordinate system with origin $(3,1)$, 
there are 6 points in $G(\sg)$ in quadrant I, 
2 points in $G(\sg)$ in quadrant II, no  points in $G(\sg)$ in quadrant III, and no  points in $G(\sg)$ in quadrant IV.  We let 
$\mmp^{(a,b,c,d)}(\sg)$ denote the number of $i$ such that 
$\sg_i$ matches the marked mesh pattern $MMP(a,b,c,d)$ in $\sg$.

\fig{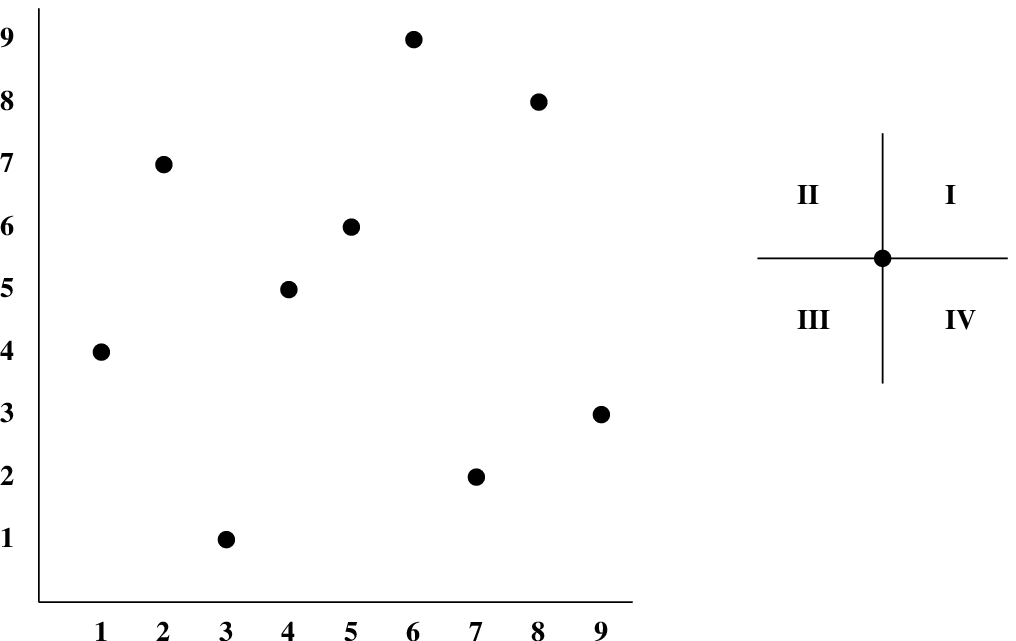}{The graph of $\sg = 471569283$.}

Note how the (two-dimensional) notation of \'Ulfarsson \cite{Ulf} for marked mesh patterns corresponds to our (one-line) notation for quadrant marked mesh patterns. For example,

\[
MMP(0,0,k,0)=\mmpattern{scale=2.3}{1}{1/1}{}{0/0/1/1/k}\hspace{-0.25cm},\  MMP(k,0,0,0)=\mmpattern{scale=2.3}{1}{1/1}{}{1/1/2/2/k}\hspace{-0.25cm},
\]

\[
MMP(0,a,b,c)=\mmpattern{scale=2.3}{1}{1/1}{}{0/1/1/2/a} \hspace{-2.07cm} \mmpattern{scale=2.3}{1}{1/1}{}{0/0/1/1/b} \hspace{-2.07cm} \mmpattern{scale=2.3}{1}{1/1}{}{1/0/2/1/c} \ \mbox{ and }\ \ \ MMP(0,0,\emptyset,k)=\mmpattern{scale=2.3}{1}{1/1}{0/0}{1/0/2/1/k}\hspace{-0.25cm}.
\]

Kitaev and Remmel \cite{kitrem} studied the distribution of 
quadrant marked mesh patterns in the symmetric group $S_n$ 
and Kitaev, Remmel, and 
Tiefenbruck \cite{kitremtie1,kitremtie2} 
studied the distribution of quadrant marked mesh patterns 
in $132$-avoiding permutations in $S_n$.
In  \cite{kitrem2}, Kitaev and Remmel studied 
the distribution of the statistics 
$\mmp^{(1,0,0,0)}$, $\mmp^{(0,1,0,0)}$, 
$\mmp^{(0,0,1,0)}$, 
and $\mmp^{(0,0,0,1)}$ in the set of up-down and down-up permutations. 
The main goal of this paper is to study the distribution 
of the statistics $\mmp^{(1,0,\emptyset,0)}$, $\mmp^{(0,1,0,\emptyset)}$, 
$\mmp^{(0,\emptyset,0,1)}$, 
and $\mmp^{(\emptyset,0,1,0)}$ in the set of up-down and down-up permutations. 
Given a permutation $\sg = \sg_1 \ldots \sg_n \in S_n$, we let 
$Des(\sg) = \{i:\sg_i > \sg_{i+1}\}$.  Then we say 
that $\sg$ is an {\em up-down permutation} if $Des(\sg)$ is the 
set of all even numbers less than or equal to $n$ and a 
{\em down-up permutation} if $Des(\sg)$ is the 
set of all odd numbers less than or equal to $n$. That is, 
$\sg$ is an up-down permutation if 
$$\sg_1 < \sg_2 > \sg_3< \sg_4 > \sg_5 < \cdots$$ 
and $\sg$ is an down-up permutation if 
$$\sg_1 > \sg_2 < \sg_3 > \sg_4 < \sg_5 > \cdots. $$ 
Let  $UD_n$ denote the set of all up-down permutations in $S_n$ and 
$DU_n$ denote the set of all down-up permutations in $S_n$. 
Given a permutation $\sg = \sg_1 \ldots \sg_n \in S_n$, 
we define the reverse of $\sg$, $\sg^r$, to be 
$\sg_n \sg_{n-1} \ldots \sg_2 \sg_1$ and the complement of 
$\sg$, $\sg^c$, to be $(n+1-\sg_1) \ldots (n+1 -\sg_n)$.

For $n \geq 1$, we let 
\begin{eqnarray*}
A^{(a,b,c,d)}_{2n}(x) &=& \sum_{\sg \in UD_{2n}}x^{\mmp^{(a,b,c,d)}(\sg)}, \ \ \ \ \ \ \ 
B^{(a,b,c,d)}_{2n-1}(x) = \sum_{\sg \in UD_{2n-1}} x^{\mmp^{(a,b,c,d)}(\sg)}, \\
C^{(a,b,c,d)}_{2n}(x) &=& \sum_{\sg \in DU_{2n}}x^{\mmp^{(a,b,d,d)}(\sg)}, \ \mbox{and} \ 
D^{(a,b,c,d)}_{2n-1}(x) = \sum_{\sg \in DU_{2n-1}}x^{\mmp^{(a,b,c,d)}(\sg)}. 
\end{eqnarray*}
We then have the following simple proposition. 
\begin{proposition}\label{prop1} For all $n \geq 1$, 

\begin{itemize} 
\item[\rm{(1)}]  $\displaystyle  A^{(a,b,c,d)}_{2n}(x) =  C^{(b,a,d,c)}_{2n}(x) = 
C^{(d,c,b,a)}_{2n}(x) = A^{(c,d,a,b)}_{2n}(x)$,\\
\item[\rm{(2)}] $\displaystyle C^{(a,b,c,d)}_{2n}(x) =  A^{(b,a,d,c)}_{2n}(x) = 
A^{(d,c,b,a)}_{2n}(x) = C^{(c,d,a,b)}_{2n}(x)$,\\
\item[\rm{(3)}] $\displaystyle
B^{(a,b,c,d)}_{2n-1}(x) = B^{(b,a,d,c)}_{2n-1}(x) = 
D^{(d,c,b,a)}_{2n-1}(x) = D^{(c,d,a,b)}_{2n-1}(x)$, and \\
\item[\rm{(4)}] $\displaystyle D^{(a,b,c,d)}_{2n-1}(x) = D^{(b,a,d,c)}_{2n-1}(x) = 
B^{(d,c,b,a)}_{2n-1}(x) = B^{(c,d,a,b)}_{2n-1}(x)$.
\end{itemize}
\end{proposition}
\begin{proof}
It is easy to see that for any 
$\sg \in S_n$, $$\mmp^{(a,b,c,d)}(\sg)=\mmp^{(b,a,d,c)}(\sg^r)= 
\mmp^{(d,c,b,a)}(\sg^c)= \mmp^{(c,d,a,b)}((\sg^r)^c).$$ 
Then part 1 easily follows since   
$$ \sg \in UD_{2n} \iff \sg^r \in DU_{2n} \iff \sg^c \in DU_{2n} \iff 
(\sg^r)^c \in UD_{2n}.$$ 

Parts 2, 3, and 4 are proved in a similar manner. 
\end{proof}

In \cite{kitrem2}, we studied the distribution of the statistics $\mmp^{(1,0,0,0)}$, $\mmp^{(0,1,0,0)}$, $\mmp^{(0,0,1,0)}$, 
and $\mmp^{(0,0,0,1)}$ in the set of up-down and down-up permutations. 
It follows from Proposition \ref{prop1} that the study the distribution 
of the statistics $\mmp^{(1,0,0,0)}$, $\mmp^{(0,1,0,0)}$, $\mmp^{(0,0,1,0)}$, 
and $\mmp^{(0,0,0,1)}$ in the set of up-down and down-up permutations 
can be reduced to the study of the following generating functions: 
\begin{eqnarray*}
A^{(1,0,0,0)}(t,x) &=& 1+ \sum_{n\geq 1} A^{(1,0,0,0)}_{2n}(x) \frac{t^{2n}}{(2n)!}, \\
B^{(1,0,0,0)}(t,x) &=& \sum_{n\geq 1} B^{(1,0,0,0)}_{2n-1}(x)\frac{t^{2n-1}}{(2n-1)!}, \\
C^{(1,0,0,0)}(t,x) &=& 1+ \sum_{n\geq 1} C^{(1,0,0,0)}_{2n}(x) \frac{t^{2n}}{(2n)!}, \ \mbox{and} \\
D^{(1,0,0,0)}(t,x) &=& \sum_{n\geq 1} D^{(1,0,0,0)}_{2n-1}(x) \frac{t^{2n-1}}{(2n-1)!}.
\end{eqnarray*} 

In the case when $x=1$, these generating functions are well known. 
That is, for any $(a,b,c,d)$, let 
$A_{2n}(1) = A^{(a,b,c,d)}_{2n}(1)$, $B_{2n+1}(1) = B^{(a,b,c,d)}_{2n+1}(1)$, 
$C_{2n}(1) = C^{(a,b,c,d)}_{2n}(1)$, and 
$D_{2n}(1) = D^{(a,b,c,d)}_{2n}(1)$.  
The operation of complementation shows that 
 $A_{2n}(1) = C_{2n}(1)$ and $B_{2n_1}(1) = D_{2n-1}(1)$ for 
all $n \geq 1$.  Andr\'{e} \cite{Andre1,Andre2}
proved that
\begin{equation}\label{sec}
1+ \sum_{n\geq 0} A_{2n}(1) \frac{t^{2n}}{(2n)!} = \sec(t)
\end{equation}
and 
\begin{equation}\label{tan}
\sum_{n\geq 1} B_{2n-1}(1) \frac{t^{2n+1}}{(2n+1)!} = \tan(t).
\end{equation}

In \cite{kitrem2}, we proved the following which can be 
viewed as a refinement of Andr\'{e}'s results.  

\begin{theorem}\label{thm:oldmain}
\begin{eqnarray}
A^{(1,0,0,0)}(t,x) &=& (\sec(xt))^{1/x}, \\
B^{(1,0,0,0)}(t,x) &=& (\sec(xt))^{1/x} \int_0^t (\sec(xz))^{-1/x} dz, \\
C^{(1,0,0,0)}(t,x)&=& 1 + \int_0^t (\sec(xy))^{1+\frac{1}{x}}\int_0^y (\sec(xz))^{1/x} dz\ dy, \ \mbox{and}\\
D^{(1,0,0,0)}(t,x) &=& \int_0^t (\sec(xz))^{1+\frac{1}{x}}dz.
\end{eqnarray}
\end{theorem}

In this paper, we prove a different refinement of Adr\'{e}'s results by 
studying  
the distribution of the statistics $\mmp^{(1,0,\emptyset,0)}$, 
$\mmp^{(\emptyset,0,1,0)}$, $\mmp^{(0,1,0,\emptyset)}$, 
and $\mmp^{(0,\emptyset,0,1)}$ in the set of up-down and down-up permutations. 
It follows from Proposition \ref{prop1} that the study the distribution 
of the statistics $\mmp^{(1,0,\emptyset,0)}$, 
$\mmp^{(\emptyset,0,1,0)}$, $\mmp^{(0,1,0,\emptyset)}$, 
and $\mmp^{(0,\emptyset,0,1)}$ in the set of up-down and down-up permutations 
can be reduced to the study of the following generating functions: 
\begin{eqnarray*}
A^{(1,0,\emptyset,0)}(t,x) &=& 1+ \sum_{n\geq 1} A^{(1,0,\emptyset,0)}_{2n}(x) \frac{t^{2n}}{(2n)!}, \\
B^{(1,0,\emptyset,0)}(t,x) &=& \sum_{n\geq 1} B^{(1,0,\emptyset,0)}_{2n-1}(x)\frac{t^{2n-1}}{(2n-1)!}, \\
C^{(1,0,\emptyset,0)}(t,x) &=& 1+ \sum_{n\geq 1} C^{(1,0,\emptyset,0)}_{2n}(x) \frac{t^{2n}}{(2n)!}, \ \mbox{and} \\
D^{(1,0,\emptyset,0)}(t,x) &=& \sum_{n\geq 1} D^{(1,0,\emptyset,0)}_{2n-1}(x) \frac{t^{2n-1}}{(2n-1)!}.
\end{eqnarray*}

The main goal of this paper is prove the following theorem. 

\begin{theorem}\label{thm:main}
\begin{eqnarray}
A^{(1,0,\emptyset,0)}(t,x) &=& (\sec(t))^x, \\
B^{(1,0,\emptyset,0)}(t,x) &=& 
\frac{\sin(t)\cos(t)(1-x+x\sec(t))}{x+(1-x)\cos(t)} \times \\
&& \ \ \ \left( 
(1-x)\ {}_2F_1 (\frac{1}{2},\frac{1+x}{2};\frac{3}{2};(\sin(t))^2) + 
x\  {}_2F_1 (\frac{1}{2},\frac{2+x}{2};\frac{3}{2};(\sin(t))^2)\right) 
\nonumber \\
D^{(1,0,\emptyset,0)}(t,x)&=& x (\sec(t))^x \int_0^t (\cos(z))^x dz +(1-x)
\int_0^t (\sec(z))^x dz, \ \mbox{and}\\
C^{(1,0,\emptyset,0)}(t,x) &=& 
1+ \int_0^t x (\sec(z))^x (1-x+x\sec(z)) \int_0^z \cos(y)dy\ dz + \\
&& \ \ \ (1-x) \int_0^t B^{(1,0,\emptyset,0)}(t,z)dz. \nonumber 
\end{eqnarray}
\end{theorem}
Here ${}_2F_1(a,b;c;z)= \sum_{n=0}^\infty \frac{(a)_n(b)_n}{(c)_n} 
\frac{z^n}{n!}$ where 
$(x)_n = x(x-1) \cdots (x-n+1)$ if $n \geq 1$ and $(x)_0 =1$.

One can use these generating functions to find some initial values 
of the polynomials $A^{(1,0,\emptyset,0)}_{2n}(x)$, $B^{(1,0,\emptyset,0)}_{2n-1}(x)$, $C^{(1,0,\emptyset,0)}_{2n}(x)$,  and 
$D^{(1,0,\emptyset,0)}_{2n-1}(x)$. For example, we have used Mathematica to compute 
 the following tables. \\
\ \\
\begin{tabular}{|l|l|}
\hline
$n$ &  $A^{(1,0,\emptyset,0)}_{2n}(x)$ \\
\hline
0 & 1 \\
\hline
1 & x\\
\hline
2 & $x (2+3x)$ \\
\hline
3 & $x \left(16+30 x+15 x^2\right)$\\
\hline
4 & $x \left(272+588 x+420 x^2+105 x^3\right)$\\
\hline
5 & $x \left(7936+18960 x+16380 x^2+6300 x^3+945
x^4\right)$\\
\hline
6 & $x \left(353792+911328 x+893640 x^2+429660 x^3+103950 x^4+
10395 x^5\right)$\\
\hline
\end{tabular} \\
\ \\
\ \\
\begin{tabular}{|l|l|}
\hline
$n$ &  $B^{(1,0,\emptyset,0)}_{2n+1}(x)$ \\
\hline 
0& 1 \\
\hline
1 & $2x$\\
\hline
2 & $x \left(7+9x\right) $ \\
\hline 
3 & $x \left(77+135x+60x^2\right)$ \\
\hline
4 & $x \left(1657+3444x+2310x^2+525x^3\right)$ \\
\hline 
5 & $x
\left(58457+ 135945x+112770x^2+40950x^3+5670x^4\right)$\\
\hline
6 & $x \left(3056557+7715664x+7347945x^2+3395700x^3+777625x^4+72765x^5\right)$\\
\hline
\end{tabular}\\
\ \\
\ \\
\begin{tabular}{|l|l|}
\hline
$n$ &  $C^{(1,0,\emptyset,0)}_{2n}(x)$ \\
\hline
0 & 1 \\
\hline
1 & 1 \\
\hline 
2 & $x \left(2+3x\right)$\\
\hline 
3 & $x \left(7+35x+19x^2\right)$ \\
\hline 
4 & $x \left(77+581x+571x^2+156x^3\right)$ \\
\hline 
5 & $x \left(1657+16428x +21066x^2+9738x^3+1587x^4\right)$\\
\hline 
6 & $x \left(58457+712579x+1079747x^2 + 652452x^3+180240x^4+19290x^5\right)$
\\
\hline 
\end{tabular}\\
\ \\
\ \\
\begin{tabular}{|l|l|}
\hline
$n$ &  $D^{(1,0,\emptyset,0)}_{2n+1}(x)$ \\
\hline
0 & 1 \\
\hline 
1 & $x (1+x)$ \\
\hline 
2 & $x \left(2+9x+5x^2\right)$ \\
\hline 
3 & $x \left(16+110x+113x^2+33x^3\right)$ \\
\hline 
4 & $x \left(272+2492x+3288x^2+1605x^3+279x^4\right)$ \\
\hline 
5 & $x \left(7936+90384x+139756x^2+87456x^3+25365x^4+2895x^5\right)$ \\
\hline 
6 & $x \left(353792+4803040x+8323816x^2+6110100x^3+2297778x^4+444045x^5+35685x^6\right)$\\
\hline
\end{tabular}
\ \\

The outline of this paper is as follows.  In section 2, we shall 
prove theorem \ref{thm:main}. Then in section 3, we shall show how 
several of the entries of the tables above can be explained. In particular, 
we will derive formulas for the coefficient of the highest and 
lowest coefficient of $x$ in the polynomials 
the polynomials 
$A^{(1,0,\emptyset,0)}_{2n}(x)$, $B^{(1,0,\emptyset,0)}_{2n+1}(x)$, 
$C^{(1,0,\emptyset,0)}_{2n}(x)$, and $D^{(1,0,\emptyset,0)}_{2n+1}(x)$ 
as well as formulas for the second highest and second lowest coefficient 
of $x$ in these polynomials. Finally, in section 3, we shall 
discuss some connections with our previous work \cite{kitrem2}  on 
quadrant marked mesh patterns in alternating permutations 
as well as some directions for further research.

\section{Proof of Theorem \ref{thm:main}}

The proof of all parts of Theorem \ref{thm:main} proceed 
in the same manner. That is, there are simple  
recursions satisfied by the polynomials 
$A^{(1,0,\emptyset,0)}_{2n}(x)$, $B^{(1,0,\emptyset,0)}_{2n+1}(x)$, 
$C^{(1,0,\emptyset,0)}_{2n}(x)$, and $D^{(1,0,\emptyset,0)}_{2n+1}(x)$ 
based on the possible positions of 1 in an up-down or 
a down-up permutation. 

\subsection{The generating function $A^{(1,0,\emptyset,0)}(t,x)$}

If $\sg = \sg_1 \ldots \sg_{2n} \in UD_{2n}$, then  
$1$ must occur in one of the positions $1,3, \ldots , 2n-1$.  Let 
$UD_{2n}^{(2k+1)}$ denote the set of permutations $\sg \in UD_{2n}$ 
such that $\sg_{2k+1} = 1$.  A schematic diagram of an element
in $UD_{2n}^{(2k+1)}$ is pictured in Figure \ref{fig:ud12n}.

\fig{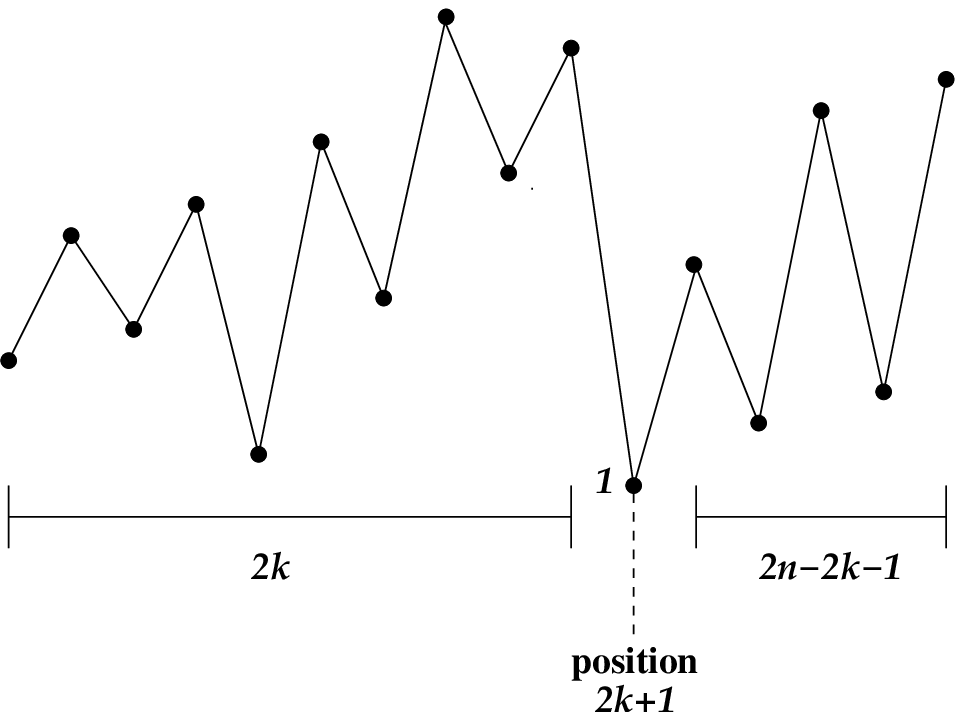}{The graph of a $\sg \in UD_{2n}^{(2k+1)}$.}

Consider a $\sg =\sg_1 \ldots \sg_{2n} \in UD_{2n}^{(2k+1)}$ 
where $0 \leq k \leq n-1$. 
Note that there are $\binom{2n-1}{2k}$ ways 
to pick the elements which occur to the right of position $2k+1$ in such a  
$\sg$ and there are  $D_{2n-2k-1}(1)= B_{2n-2k-1}(1)$ ways to 
order them since 
the elements to the right of position $2k+1$ must form a 
down-up permutation of length $2n-2k-1$. The fact that 
$\sg_{2k+1} = 1$ implies that $\sg_{2k+1}$ matches 
 $MMP(1,0,\emptyset,0)$ in $\sg$ and that 
none of the elements to the right of position $2k+1$ match 
 $MMP(1,0,\emptyset,0)$ in $\sg$. 
Thus the contribution of the elements to 
the right of position $2k+1$ in 
$\sum_{\sg \in UD_{2n}^{(2k+1)}} x^{\mmp^{(1,0,\emptyset,0)}(\sg)}$ 
is $B_{2n-2k-1}(1)$. Now the only 
possible elements of   $\sg_1, \ldots, \sg_{2k}$ that 
can contribute to  $\mmp^{(1,0,\emptyset,0)}(\sg)$ are 
$\sg_1, \sg_3, \ldots, \sg_{2k-1}$.  Since each of the elements 
have an element to its right in $\sg_1 \ldots \sg_{2k}$ which 
is larger than that element, it follows that the 
elements to the right of position $2k+1$ have no effect 
on whether $\sg_1, \ldots, \sg_{2k}$ can contribute 
to $\mmp^{(1,0,\emptyset,0)}(\sg)$. 
Hence the contribution 
of the elements to the left of position $2k+1$ in 
$\sum_{\sg \in UD_{2n}^{(2k+1)}} x^{\mmp^{(1,0,\emptyset,0)}(\sg)}$ is 
$A^{(1,0,\emptyset,0)}_{2k}(x)$. It thus follows that for $n \geq 1$, 
\begin{equation*}\label{Arec1}
A^{(1,0,\emptyset,0)}_{2n}(x) =  
x \sum_{k=0}^{n-1} \binom{2n-1}{2k} B_{2n-2k-1}(1) A^{(1,0,\emptyset,0)}_{2k}(x)
\end{equation*}
or, equivalently,  
\begin{equation}\label{Arec2}
\frac{A^{(1,0,\emptyset,0)}_{2n}(x)}{(2n-1)!} = 
x\sum_{k=0}^{n-1} \frac{B_{2n-2k-1}(1)}{(2n-2k-1)!} 
\frac{A^{(1,0,\emptyset,0)}_{2k}(x)}{(2k)!}.
\end{equation}
Multiplying both sides of (\ref{Arec2}) by $t^{2n-1}$ and summing 
for $n \geq 1$, we see that 
\begin{equation*}
\sum_{n \geq 1} \frac{A^{(1,0,\emptyset,0)}_{2n}(x)t^{2n-1}}{(2n-1)!} = 
x\left(\sum_{k \geq 1}  \frac{B_{2k-1}(1)t^{2k-1}}{(2k-1)!}\right) 
\left(\sum_{k \geq 0} \frac{A^{(1,0,\emptyset,0)}_{2k}(x)t^{2k}}{(2k)!}\right).
\end{equation*} 
By (\ref{tan}), 
$$\sum_{k \geq 1}  \frac{B_{2k-1}(1)t^{2k-1}}{(2k-1)!} = 
\tan(t)
$$
so that 
\begin{equation}\label{Adiff}
\frac{\partial}{\partial t} A^{(1,0,\emptyset,0)}(t,x) =x\tan(t) A^{(1,0,\emptyset,0)}(t,x).
\end{equation}
Our initial condition is that $A^{(1,0,\emptyset,0)}(0,x)=1$.  
It is easy to check 
that the solution to this differential equation is 
\begin{equation*}\label{Afin}
A^{(1,0,\emptyset,0)}(t,x) = (\sec(t))^{x}.
\end{equation*}

\subsection{The generating function $B^{(1,0,\emptyset,0)}(t,x)$}

If $\sg = \sg_1 \ldots \sg_{2n+1} \in UD_{2n+1}$, then  
$1$ must occur in one of the positions $1,3, \ldots , 2n+1$.  Let 
$UD_{2n+1}^{(2k+1)}$ denote the set of permutations $\sg \in UD_{2n+1}$ 
such that $\sg_{2k+1} = 1$.  A schematic diagram of an element
in $UD_{2n}^{(2k+1)}$ is pictured in Figure \ref{fig:ud12n+1}.

\fig{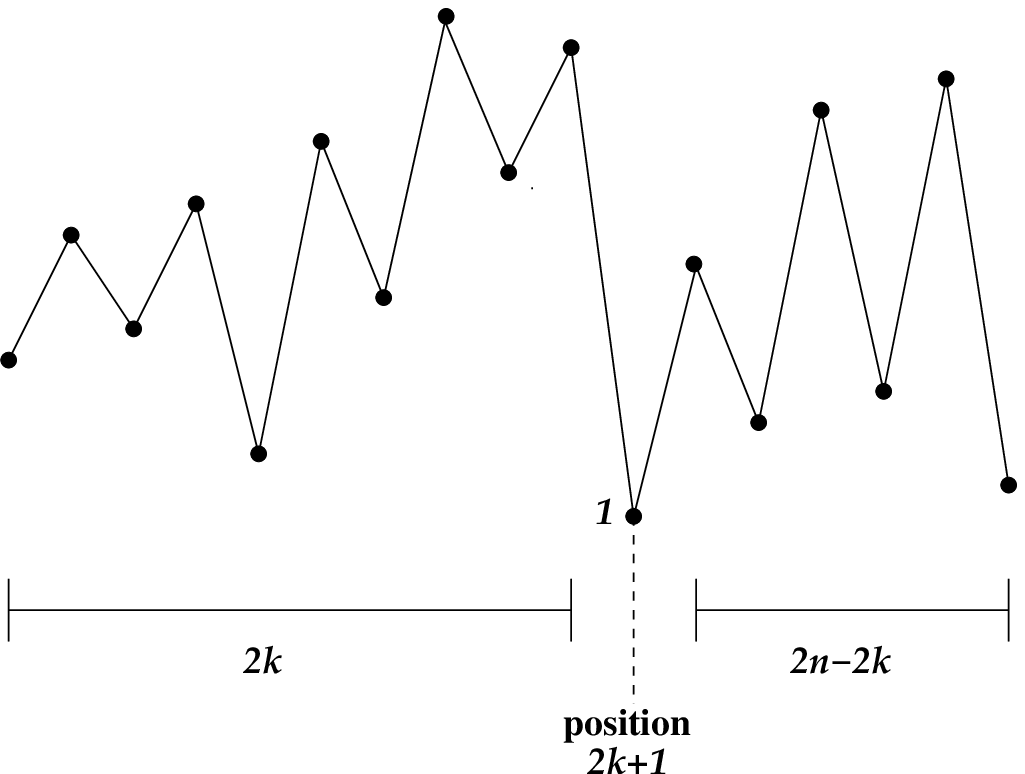}{The graph of a $\sg \in UD_{2n+1}^{(2k+1)}$.}

A permutation $\sg =\sg_1 \ldots \sg_{2n+1} \in UD_{2n+1}^{(2n+1)}$ ends 
with 1 so that $\sg_{2n+1} =1$ does not match  
$MMP(1,0,\emptyset,0)$ in $\sg$.  Moreover, 
$\red[\sg_1 \ldots \sg_{2n}] \in UD_{2n}$ and  
$\sg_{2n+1} =1$ cannot effect whether any of the other elements 
in $\sg$ match  $MMP(1,0,\emptyset,0)$. Thus 
$$ \sum_{\sg \in UD_{2n+1}^{(2n+1)}} x^{\mmp^{(1,0,\emptyset,0)}(\sg)} 
= A^{(1,0,\emptyset,0)}_{2n}(x).$$ 

Next consider $UD_{2n+1}^{(2k+1)}$ where $0 \leq k \leq n-1$. 
Note that there are $\binom{2n}{2k}$ ways 
to pick the elements which occur to the right of position $2k+1$ in such 
a $\sg$ and there are $C_{2n-2k}(1) = A_{2n-2k}(1)$ ways to order them since 
the elements to the left of position $2k$ form a down-up permutation 
of length $2k$. 
That is, the fact that 
$\sg_{2k+1} = 1$ implies that $\sg_{2k+1}$ matches 
 $MMP(1,0,\emptyset,0)$ in $\sg$ and that 
none of the elements to the right of position match 
 $MMP(1,0,\emptyset,0)$ in $\sg$.  
Thus the contribution of the elements to 
the right of position $2k+1$ in 
$\sum_{\sg \in UD_{2n+1}^{(2k+1)}} x^{\mmp^{(1,0,\emptyset,0)}(\sg)}$ 
is $C_{2n-2k}(1)= A_{2n-2k}(1)$ since 
the elements to the right of position $2k+1$ must form a 
down-up permutation of length $2n-2k$.  As we proved 
above,  
the elements to the right of position $2k$ have no effect 
on whether $\sg_1, \ldots, \sg_{2k}$ can contribute 
to $\mmp^{(1,0,\emptyset,0)}(\sg)$. It 
follows that the contribution 
of the elements to the left of position $2k+1$ in 
$\sum_{\sg \in UD_{2n+1}^{(2k+1)}} x^{\mmp^{(1,0,\emptyset,0)}(\sg)}$ is 
$A^{(1,0,\emptyset,0)}_{2k}(x)$. It thus follows that for $n \geq 1$, 
\begin{equation*}\label{Brec1}
B^{(1,0,\emptyset,0)}_{2n+1}(x) = A^{(1,0,\emptyset,0)}_{2n}(x) + 
x \sum_{k=0}^{n-1} \binom{2n}{2k} A_{2n-2k}(1) A^{(1,0,\emptyset,0)}_{2k}(x)
\end{equation*}
or, equivalently,  
\begin{equation}\label{Brec2}
\frac{B^{(1,0,\emptyset,0)}_{2n+1}(x)}{(2n)!} = 
\frac{A^{(1,0,\emptyset,0)}_{2n}(x)}{(2n)!} +
x \sum_{k=0}^{n-1} \frac{A_{2n-2k}(1)}{(2n-2k)!} 
\frac{A^{(1,0,\emptyset,0)}_{2k}(x)}{(2k)!}.
\end{equation}
Note that $B^{(1,0,\emptyset,0)}_1(x) =1$. 
Multiplying both sides of (\ref{Brec2}) by $t^{2n}$  and summing 
for $n \geq 1$, we see that 
$$\sum_{n \geq 0} \frac{B^{(1,0,\emptyset,0)}_{2n+1}(x)t^{2n}}{(2n)!} = 
\sum_{n \geq 0} \frac{A^{(1,0,\emptyset,0)}_{2n}(x)t^{2n}}{(2n)!} + 
x\left(\sum_{k \geq 1}  \frac{A_{2k}(1)t^{2k}}{(2k)!}\right) 
\left(\sum_{k \geq 0} \frac{A^{(1,0,\emptyset,0)}_{2k}(x)t^{2k}}{(2k)!}\right).$$ 
By (\ref{sec}), 
$$\sum_{k \geq 0}  \frac{A_{2k}(1)t^{2k}}{(2k)!} = 
\sec(t)$$ 
so that 
\begin{equation*}
\frac{\partial}{\partial t} B^{(1,0,\emptyset,0)}(t,x) = (\sec(t))^x + 
x(\sec(t))^x(\sec(t) -1)).
\end{equation*}
Thus
\begin{equation}\label{Bdiff}
\frac{\partial}{\partial t} B^{(1,0,\emptyset,0)}(t,x) = 
(\sec(t))^x(1-x + x\sec(t)).
\end{equation}
Our initial condition is that $B^{(1,0,\emptyset,0)}(0,x) =0$. 
We used Mathematica to solve this differential equation which 
gave the following formula for $B^{(1,0,\emptyset,0)}(t,x)$: 
\begin{eqnarray*}\label{Bfin}
B^{(1,0,\emptyset,0)}(t,x) &=& \frac{\sin(t)\cos(t)(1-x+x\sec(t))}{x+(1-x)\cos(t)} \times \\
&& \ \ \ \left( (1-x)\  {}_2F_1 \left(\frac{1}{2},\frac{1+x}{2};\frac{3}{2};(\sin(t))^2\right) + x\ {}_2F_1\left(\frac{1}{2},\frac{2+x}{2};\frac{3}{2};(\sin(t))^2\right)\right). 
 \nonumber
\end{eqnarray*}

\subsection{The generating function $D^{(1,0,\emptyset,0)}(t,x)$}

 If $\sg = \sg_1 \ldots \sg_{2n+1} \in DU_{2n+1}$, then  
$1$ must occur in one of the positions $2,4, \ldots , 2n$.  Let 
$DU_{2n+1}^{(2k)}$ denote the set of permutations $\sg \in DU_{2n+1}$ 
such that $\sg_{2k} = 1$.  A schematic diagram of an element
in $DU_{2n+1}^{(2k)}$ is pictured in Figure \ref{fig:du12n+1}.

\fig{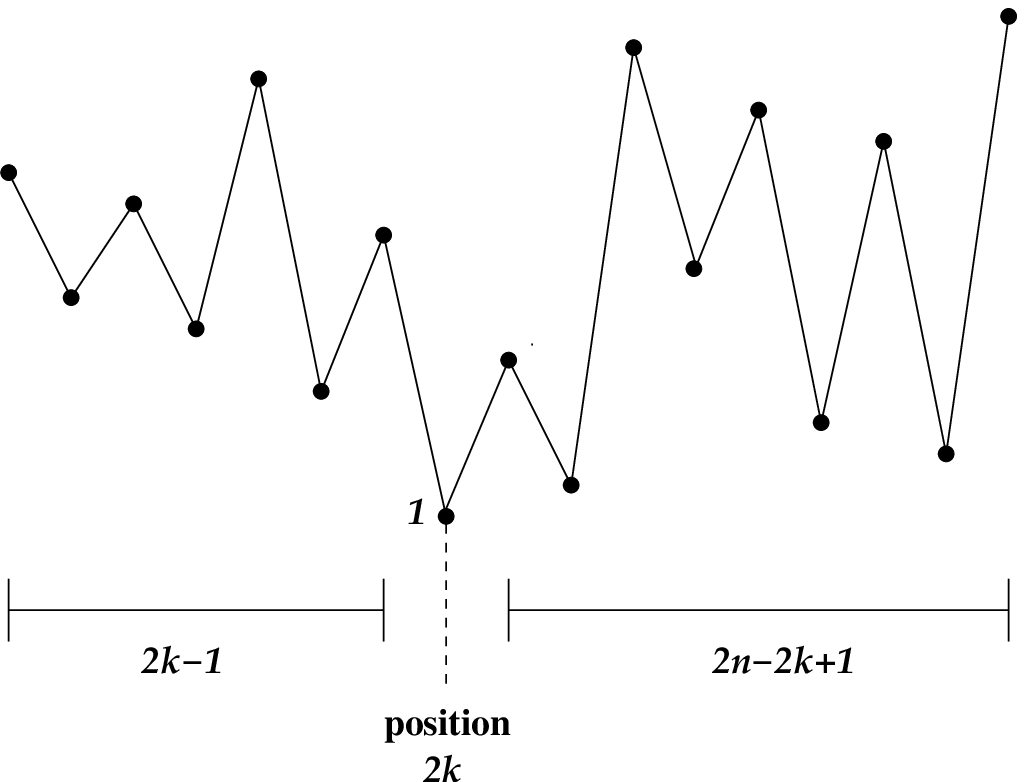}{The graph of a $\sg \in DU_{2n+1}^{(2k)}$.}

Let  
$$D^{(1,0,\emptyset,0)}_{2n+1}(x,y) = \sum_{\sg \in DU_{2n+1}} x^{\mmp^{(1,0,\emptyset,0)}(\sg)}y^{\chi(\sg_1 =2n+1)}.$$
First, we want to study the polynomial 
$\overline{D}^{(1,0,\emptyset,0)}_{2n+1}(x) =  
D^{(1,0,\emptyset,0)}_{2n+1}(x,x)$. 
Suppose that $\sg = \sg_1 \ldots \sg_{2n+1}$ is an 
element of $DU_{2n+1}$. If  
$\sg_1 < 2n+1$, then $\sg_1$ will automatically contribute to 
$\mmp^{(1,0,\emptyset,0)}(\sg)$.  However, if 
$\sg_1 = 2n+1$, then $\sg_1$ will not contribute to 
$\mmp^{(1,0,\emptyset,0)}(\sg)$. Thus the difference between 
$\overline{D}^{(1,0,\emptyset,0)}_{2n+1}(x)$ and $D_{2n+1}^{(1,0,\emptyset,0)}(x)$ is that 
$\sg_1$ always contributes a factor of 
$x$ to $x^{\mmp^{(1,0,\emptyset,0)}(\sg)}x^{\chi(\sg_1 =2n+1)}$.

First we shall prove a simple recursion for $\overline{D}^{(1,0,\emptyset,0)}_{2n+1}(x)$. That is, consider a $\sg  = \sg_1 \ldots \sg_{2n+1} \in 
DU_{2n+1}^{(2k)}$ where $1 \leq  k \leq n$. 
Note that there are $\binom{2n}{2k-1}$ ways 
to pick the elements which occur to the right of position $2k$ in such a
$\sg$ and there are $D_{2n-2k+1}(1) = B_{2n-2-1}(1)$ ways to order them since 
the elements to the right of position $2k$ form a down-up permutation 
of length $2n-2k+1$. For a $\sg = \sg_1 \ldots \sg_{2n+1} \in 
DU_{2n+1}^{(2k)}$, none of the elements $\sg_i$ for $i > 2k$  
matches  $MMP(1,0,\emptyset,0)$ in $\sg$ and 
$\sg_{2k} =1$ always matches   $MMP(1,0,\emptyset,0)$ in $\sg$.
Thus the only other elements of $\sg$ that  
can possibly contribute to 
$\mmp^{(1,0,\emptyset,0)}(\sg)x^{\chi(\sg_1=2n+1)}$ are the elements 
$\sg_1, \sg_2, \sg_4, \ldots, \sg_{2k-2}$. 
Since in $\overline{D}^{(1,0,\emptyset,0)}_{2n+1}(x)$, 
$\sg_1$ always contributes 
to $\mmp^{(1,0,\emptyset,0)}(\sg)x^{\chi(\sg_1=2n+1)}$ and 
the elements to the right of position $2k$ have no effect 
on whether $\sg_2, \ldots, \sg_{2k-2}$ contribute 
to $\mmp^{(1,0,\emptyset,0)}(\sg)$, it follows 
that  
the contribution of the elements to 
the left of position $2k$ to 
$$\sum_{\sg \in DU_{2n+1}^{(2k)}} x^{\mmp^{(1,0,\emptyset,0)}(\sg)}
x^{\chi(\sg_1=2n+1)}$$ 
is $\overline{D}_{2k-1}(x)$. 
Hence for $n \geq 1$, 
\begin{equation*}\label{Drec1}
\overline{D}^{(1,0,\emptyset,0)}_{2n+1}(x) =  
x \sum_{k=1}^n \binom{2n}{2k-1} \overline{D}^{(1,0,\emptyset,0)}_{2k-1}(x) B_{2n-2k+1}(1)
\end{equation*}
or, equivalently,  
\begin{equation}\label{Drec2}
\frac{\overline{D}^{(1,0,\emptyset,0)}_{2n+1}(x)}{(2n)!} = 
x\sum_{k=1}^n \frac{\overline{D}^{(1,0,\emptyset,0)}_{2k-1}(x)}{(2k-1)!}
\frac{B_{2n-2k+1}(1)}{(2n-2k+1)!}. 
\end{equation}
Note that $\overline{D}^{(1,0,\emptyset,0)}_1(x) =x$. 
Multiplying both sides of (\ref{Drec2}) by $t^{2n}$  and summing 
for $n \geq 1$, we see that 
$$\sum_{n \geq 1} \frac{\overline{D}^{(1,0,\emptyset,0)}_{2n+1}(x)t^{2n}}{(2n)!} -x = 
x\left(\sum_{k \geq 0}  \frac{\overline{D}^{(1,0,\emptyset,0)}_{2n+1}(x)t^{2k}}{(2k)!}\right) 
\left(\sum_{k \geq 0} \frac{B^{(1,0,\emptyset,0)}_{2k+1}(1)t^{2k+1}}{(2k+1)!}\right).$$ 
By (\ref{tan}), 
$$\sum_{k \geq 1}  \frac{B_{2k-1}(1)t^{2k-1}}{(2k-1)!} = \tan(t) 
$$ 
so that 
\begin{equation*}
\frac{\partial}{\partial t} \overline{D}^{(1,0,\emptyset,0)}(t,x) = 
x+ 
x \tan(x) \overline{D}^{(1,0,\emptyset,0)}(t,x).
\end{equation*}
Our initial condition is that $\overline{D}^{(1,0,\emptyset,0)}(0,x) =0$. 
One can easily check that the solution to this differential equation is 
\begin{equation}\label{oDfin}
\overline{D}^{(1,0,\emptyset,0)}(t,x) = x (\sec(t))^x \int_0^t (\cos(z))^xdz.
\end{equation}

As observed above, the difference between 
$D_{2n+1}^{(1,0,\emptyset,0)}(x)$ and 
$\overline{D}_{2n+1}^{(1,0,\emptyset,0)}(x)$ is 
that the permutations $\sg = \sg_1 \ldots \sg_{2n+1} \in DU_{2n+1}$ 
such that $\sg_1 =2n+1$ are weighted differently in that such 
permutations are weighted with an extra power of $x$ in 
$\overline{D}_{2n+1}^{(1,0,\emptyset,0)}(x)$ than they are 
in $D_{2n+1}^{(1,0,\emptyset,0)}(x)$. 
That is, 
$$ x \sum_{\sg \in DU_{2n+1},\sg_1 =2n+1} x^{\mmp^{(1,0,\emptyset,0)}(\sg)}
= \sum_{\sg \in DU_{2n+1},\sg_1 =2n+1} 
x^{\mmp^{(1,0,\emptyset,0)}(\sg)} x^{\chi(\sg_1=2n+1)}.$$ 
It is easy to see 
that 
$$ \sum_{\sg \in DU_{2n+1},\sg_1 =2n+1} x^{\mmp^{(1,0,\emptyset,0)}(\sg)}
= A^{(1,0,\emptyset,0)}_{2n}(x).$$
Thus it follows that 
\begin{equation}\label{DDrec1}
D_{2n+1}^{(1,0,\emptyset,0)}(x) = \overline{D}_{2n+1}^{(1,0,\emptyset,0)}(x)
+ (1-x) A^{(1,0,\emptyset,0)}_{2n}(x).
\end{equation}
Multiplying both sides of (\ref{DDrec1}) by $\frac{t^{2n+1}}{(2n+1)!}$ 
and summing for $n \geq 0$, we see 
that 
$$D^{(1,0,\emptyset,0)}(t,x) = \overline{D}^{(1,0,\emptyset,0)}(t,x) + 
(1-x) \int_0^t 
A^{(1,0,\emptyset,0)}(z,x)dz.$$
Hence, 
\begin{equation*}\label{Dfin}
D^{(1,0,\emptyset,0)}(t,x) = 
x (\sec(t))^x \int_0^t (\cos(z))^x dz+ (1-x) \int_0^t (\sec(z))^x dz.
\end{equation*}

\subsection{The generating function $C^{(1,0,\emptyset,0)}(t,x)$}

 If $\sg = \sg_1 \ldots \sg_{2n} \in DU_{2n}$, then  
$1$ must occur in one of the positions $2,4, \ldots , 2n$.  Let 
$DU_{2n}^{(2k)}$ denote the set of permutations $\sg \in DU_{2n}$ 
such that $\sg_{2k} = 1$.  A schematic diagram of an element
in $DU_{2n}^{(2k)}$ is pictured in Figure \ref{fig:du12n}.

\fig{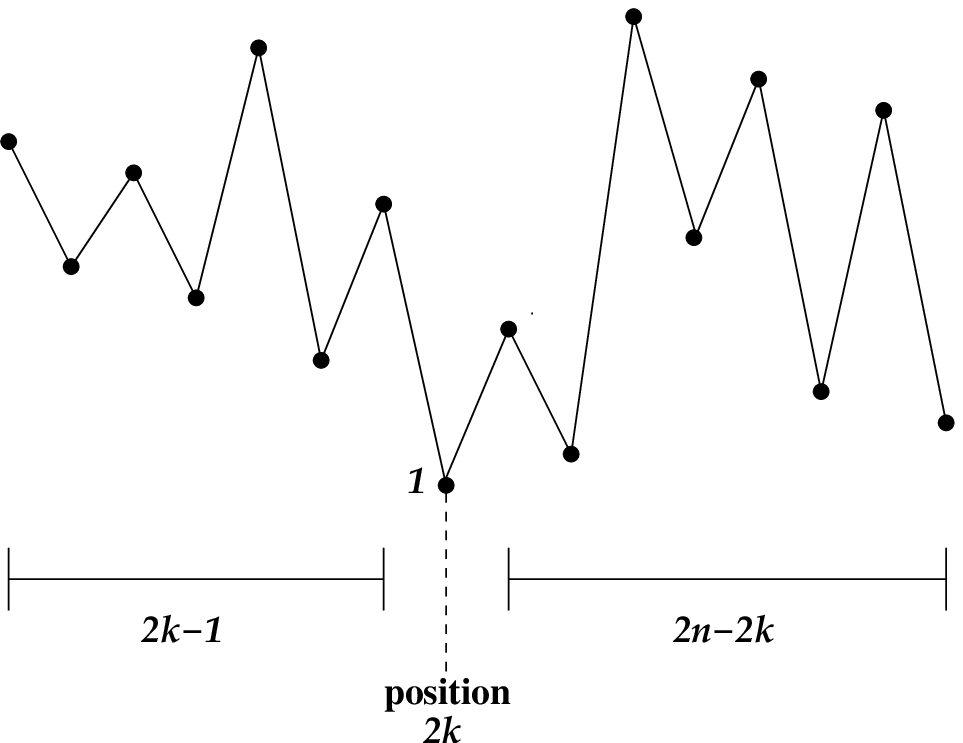}{The graph of a $\sg \in DU_{2n}^{(2k)}$.}

Let  
$$C^{(1,0,\emptyset,0)}_{2n}(x,y) = \sum_{\sg \in DU_{2n}} x^{\mmp^{(1,0,\emptyset,0)}(\sg)}y^{\chi(\sg_1 =2n)}.$$
First, we want to study the polynomial 
$\overline{C}^{(1,0,\emptyset,0)}_{2n}(x) =  
C^{(1,0,\emptyset,0)}_{2n+1}(x,x)$. 
As was the case with $D_{2n+1}^{(1,0,\emptyset,0)}(x)$, if 
$\sg = \sg_1 \ldots \sg_{2n} \in DU_{2n}$  and 
$\sg_1 < 2n$, then $\sg_1$ will automatically contribute to 
$\mmp^{(1,0,\emptyset,0)}(\sg)$.  However, if 
$\sg_1 = 2n$, then $\sg_1$ will not contribute to 
$\mmp^{(1,0,\emptyset,0)}(\sg)$. Thus the difference between 
$\overline{C}^{(1,0,\emptyset,0)}_{2n}(x)$ and $C_{2n}^{(1,0,\emptyset,0)}(x)$ is that 
$\sg_1$ always contributes a factor of 
$x$ to $x^{\mmp^{(1,0,\emptyset,0)}(\sg)}x^{\chi(\sg_1 =2n)}$.

First consider $\sg = \sg_1 \ldots \sg_{2n} \in DU_{2n}^{(2n)}$. 
Since $\sg_{2n}=1$, it is easy to see that 
$$\sum_{\sg \in DU_{2n}^{(2n)}} x^{\mmp^{(1,0,\emptyset,0)}(\sg)} 
x^{\chi(\sg_1 =2n)} = \overline{D}^{(1,0,\emptyset,0)}_{2n-1}(x).$$

Next consider a $\sg = \sg_1 \ldots \sg_{2n} \in DU_{2n}^{(2k)}$ 
where $1 \leq  k < n$. 
Note that there are $\binom{2n-1}{2k-1}$ ways 
to pick the elements which occur to the right of position $2k$ in such 
a $\sg$ and there are $C_{2n-2k}(1) = A_{2n-2k}(1)$ ways to order them since 
the elements to the right of position $2k$ form a down-up permutation 
of length $2n-2k$. For a $\sg = \sg_1 \ldots \sg_{2n+1} \in 
DU_{2n}^{(2k)}$, none of the elements $\sg_i$ for $i > 2k$  
matches  $MMP(1,0,\emptyset,0)$ in $\sg$ and 
$\sg_{2k} =1$ always matches   $MMP(1,0,\emptyset,0)$ in $\sg$.
Thus the only other elements of $\sg$ that  
can possibly contribute to 
$\mmp^{(1,0,\emptyset,0)}(\sg)x^{\chi(\sg_1=2n)}$ are the elements 
$\sg_1, \sg_2, \sg_4, \ldots, \sg_{2k-2}$. 
Since in $\overline{D}^{(1,0,\emptyset,0)}_{2n+1}(x)$, 
$\sg_1$ always contributes 
to $\mmp^{(1,0,\emptyset,0)}(\sg)x^{\chi(\sg_1=2n)}$ and 
the elements to the right of position $2k$ have no effect 
on whether $\sg_2, \ldots, \sg_{2k-2}$ contribute 
to $\mmp^{(1,0,\emptyset,0)}(\sg)$, it follows 
that  
the contribution of the elements to 
the left of position $2k$ to  
$\sum_{\sg \in DU_{2n}^{(2k)}} x^{\mmp^{(1,0,\emptyset,0)}(\sg)}
x^{\chi(\sg_1=2n)}$ 
is $\overline{D}_{2k-1}(x)$. 
Hence for $n \geq 1$, 
\begin{equation*}\label{Crec1}
\overline{C}^{(1,0,\emptyset,0)}_{2n}(x) =  \overline{D}_{2n-1}^{(1,0,\emptyset,0)}(x) + 
x \sum_{k=1}^{n-1} \binom{2n-1}{2k-1} \overline{D}^{(1,0,\emptyset,0)}_{2k-1}(x) 
A_{2n-2k}(1)
\end{equation*}
or, equivalently,  
\begin{equation}\label{Crec2}
\frac{\overline{C}^{(1,0,\emptyset,0)}_{2n}(x)}{(2n-1)!} = 
\frac{\overline{D}_{2n-1}^{(1,0,\emptyset,0)}(x)}{(2n-1)!}+
x\sum_{k=1}^{n-1} \frac{\overline{D}^{(1,0,\emptyset,0)}_{2k-1}(x)}{(2k-1)!}
\frac{A_{2n-2k}(1)}{(2n-2k)!}. 
\end{equation}

Multiplying both sides of (\ref{Crec2}) by $t^{2n-1}$ and summing 
for $n \geq 1$, we see that 
\begin{eqnarray*}
\sum_{n \geq 1} \frac{\overline{C}^{(1,0,\emptyset,0)}_{2n}(x) t^{2n-1}}{(2n-1)!} &=& 
\sum_{n \geq 1} \frac{\overline{D}^{(1,0,\emptyset,0)}_{2n-1}(x) t^{2n-1}}{(2n-1)!}
+ \\
&&  x\left(\sum_{k \geq 0}  
\frac{\overline{D}^{(1,0,\emptyset,0)}_{2n+1}(x) t^{2n+1}}{(2n+1)!}\right)  
\left( \sum_{k \geq 1} \frac{A_{2k}(1)t^{2k}}{(2k)!}\right).
\end{eqnarray*}
By (\ref{sec}), 
$$\sum_{k \geq 0}  \frac{A_{2k}(1)t^{2k}}{(2k)!} = \sec(t) 
$$ 
so that 
\begin{eqnarray*}
\frac{\partial}{\partial t} \overline{C}^{(1,0,\emptyset,0)}(t,x) &=&  
\overline{D}^{(1,0,\emptyset,0)}(t,x) +  
x \overline{D}^{(1,0,\emptyset,0)}(t,x)(\sec(t) -1) \\
&=& \overline{D}^{(1,0,\emptyset,0)}(t,x)(1-x +x \sec(t)) \\
&=& x (\sec(t))^x (1-x +x \sec(t))\int_0^t (\cos(z))^x dz.
\end{eqnarray*}
Our initial condition is that $\overline{C}^{(1,0,\emptyset,0)}(0,x) =1$. 
Maple will give a solution to this differential equation, but 
it is a complicated expression which is not particularly useful 
for our purposes so that we will simply record the solution 
to this differential equation as 
\begin{equation}\label{oCfin}
\overline{C}^{(1,0,\emptyset,0)}(t,x) = 1+ 
\int_0^t x (\sec(z))^x (1-x +x \sec(z))\int_0^z (\cos(y))^x dy \ dz.
\end{equation}

As observed above, the difference between 
$C_{2n}^{(1,0,\emptyset,0)}(x)$ and 
$\overline{C}_{2n}^{(1,0,\emptyset,0)}(x)$ is 
that the permutations $\sg = \sg_1 \ldots \sg_{2n+1} \in DU_{2n}$ 
such that $\sg_1 =2n$ are weighted differently in that such 
permutations are weighted with an extra power of $x$ in 
$\overline{C}_{2n}^{(1,0,\emptyset,0)}(x)$ than they are 
in $C_{2n}^{(1,0,\emptyset,0)}(x)$. 
That is, 
$$ x \sum_{\sg \in DU_{2n},\sg_1 =2n} x^{\mmp^{(1,0,\emptyset,0)}(\sg)}
= \sum_{\sg \in DU_{2n},\sg_1 =2n} 
x^{\mmp^{(1,0,\emptyset,0)}(\sg)} x^{\chi(\sg_1=2n)}.$$ 
It is easy to see 
that 
$$ \sum_{\sg \in DU_{2n},\sg_1 =2n} x^{\mmp^{(1,0,\emptyset,0)}(\sg)}
= B^{(1,0,\emptyset,0)}_{2n-1}(x).$$
Thus it follows that 
\begin{equation}\label{CCrec1}
C_{2n}^{(1,0,\emptyset,0)}(x) = \overline{C}_{2n}^{(1,0,\emptyset,0)}(x)
+ (1-x) B^{(1,0,\emptyset,0)}_{2n-1}(x).
\end{equation}
Multiplying both sides of (\ref{CCrec1}) by $\frac{t^{2n}}{(2n)!}$ 
and summing for $n \geq 0$, we see 
that 
$$ C^{(1,0,\emptyset,0)}(t,x) = 
\overline{C}^{(1,0,\emptyset,0)}(t,x) + (1-x) \int_0^t 
B^{(1,0,\emptyset,0)}(z,x)dz.$$
Hence, 
\begin{eqnarray*}\label{Cfin}
C^{(1,0,\emptyset,0)}(t,x) &=& 1+ 
\int_0^t x (\sec(z))^x (1-x +x \sec(z))\int_0^z (\cos(y))^x dy \ dz + \\
&& (1-x) \int_0^t B^{(1,0,\emptyset,0)}(z,x)dz.
\end{eqnarray*}

\section{The coefficients of the polynomials $A^{(1,0,\emptyset,0)}_{2n}(x)$,
\\ 
$B^{(1,0,\emptyset,0)}_{2n+1}(x)$, $C^{(1,0,\emptyset,0)}_{2n}(x)$, and 
$D^{(1,0,\emptyset,0)}_{2n+1}(x)$.}

The main goal of this section is to explain several of the coefficients of the polynomials $A^{(1,0,\emptyset,0)}_{2n}(x)$, $B^{(1,0,\emptyset,0)}_{2n+1}(x)$, $C^{(1,0,\emptyset,0)}_{2n}(x)$, and $D^{(1,0,\emptyset,0)}_{2n+1}(x)$.  
For $n \geq 1$, let $(2n)!! = \prod_{i=1}^n 2i$ and 
 $(2n-1)!! = \prod_{i=1}^n (2i-1)$. 
First it is easy to understand the coefficients of the 
lowest power of $x$ in each of these polynomials. That is, 
we have the following theorem. 

\begin{theorem}\label{lowest}\

\begin{itemize} 
\item[\rm{(1)}] For all $n \geq 1$, 
\begin{equation*} 
A^{(1,0,\emptyset,0)}_{2n}(x)|_{x} = B_{2n-1}(1).
\end{equation*}
\item[\rm{(2)}]  For all $n \geq 1$, 
\begin{equation*} 
B^{(1,0,\emptyset,0)}_{2n+1}(x)|_{x} = A_{2n}(1)+B_{2n-1}(1).\end{equation*}
\item[\rm{(3)}]  For all $n \geq 2$, 
\begin{equation*} 
C^{(1,0,\emptyset,0)}_{2n}(x)|_{x^k} = A_{2n-2}(1)+B_{2n-3}(1).
\end{equation*}
\item[\rm{(4)}]  For all $n \geq 1$, 
\begin{equation*} 
D^{(1,0,\emptyset,0)}_{2n+1}(x)|_{x^k} = B_{2n-1}(1).
\end{equation*}
\end{itemize}

\end{theorem}
\begin{proof}

For (1), note that if $\sg = \sg_1 \ldots \sg_{2n} \in UD_{2n}$ where 
$n \geq 1$, then 
$\sg_{1}$ always matches $MMP(1,0,\emptyset,0)$ in $\sg$.  
Moreover if $\sg_1 \neq 1$, then $\sg_{2k+1} =1$  for some 
$k \geq 1$ in which case $\sg_{2k+1}$ will also match 
$MMP(1,0,\emptyset,0)$ in $\sg$. Thus the only possible way to 
have $\mmp^{(1,0,\emptyset,0)}(\sg) =1$ is if $\sg_1 =1$ in which case none of 
$\sg_2, \ldots, \sg_{2n}$ will match  $MMP(1,0,\emptyset,0)$ in $\sg$. Clearly in such a situation, $\red[\sg_2 \ldots \sg_{2n}] \in DU_{2n-1}$ so that we have $D_{2n-1}(1) = B_{2n-1}(1)$ ways to choose 
$\sg_2 \ldots \sg_{2n}$. It follows that  
$A^{(1,0,\emptyset,0)}_{2n}(x)|_{x} = B_{2n-1}(1)$ for $n \geq 1$.

For (2), note that if $\sg = \sg_1 \ldots \sg_{2n+1} \in UD_{2n+1}$ 
where $n \geq 1$, then again $\sg_{1}$ always matches $MMP(1,0,\emptyset,0)$ in $\sg$.  
However, in this case, 
if $\sg_1 \neq 1$, then $\sg_{2k+1} =1$  for some $k \geq 1$ in 
which case $\sg_{2k+1}$ will also match 
$MMP(1,0,\emptyset,0)$ in $\sg$ if $1 \leq k < n$, but 
it will not match $MMP(1,0,\emptyset,0)$ in $\sg$ if $k=n$. 
Thus if $\sg = \sg_1 \ldots \sg_{2n+1} \in UD_{2n+1}$  and 
$\mmp^{(1,0,\emptyset,0)}(\sg) =1$, it must be the case that 
$\sg_1 =1$ or $\sg_{2n+1} =1$. 
Now if $\sg_1 =1$, then none of 
$\sg_2, \ldots, \sg_{2n+1}$ will match  $MMP(1,0,\emptyset,0)$ in $\sg$. Clearly in such a situation, $\red[\sg_2 \ldots \sg_{2n+1}] \in DU_{2n}$ so that we have $C_{2n}(1) = A_{2n}(1)$ ways to choose 
$\sg_2 \ldots \sg_{2n+1}$. If $\sg_{2n+1} =1$, then 
it must be the case that $\sg_1 =2$ since otherwise $\sg_{2k+1} =2$ 
for some $k \geq 1$ in which case 
$\sg_1$ and $\sg_{2k+1} =2$ will match  $MMP(1,0,\emptyset,0)$ in $\sg$. But then $\red[\sg_2 \ldots \sg_{2n}] \in DU_{2n-1}$ so that we have $D_{2n-1}(1) = B_{2n-1}(1)$ ways to choose 
$\sg_2 \ldots \sg_{2n}$. Hence  
$B^{(1,0,\emptyset,0)}_{2n+1}(x)|_{x} = A_{2n}(1)+B_{2n-1}(1)$ for $n \geq 1$.

For (3), note that if $\sg = \sg_1 \ldots \sg_{2n} \in DU_{2n}$ where 
$n \geq 2$, then 
$\sg_{2}$ always matches $MMP(1,0,\emptyset,0)$ in $\sg$.  Moreover, 
if $\sg_2 \neq 1$, then $\sg_{2k} =1$ for some $k \geq 2$ in which 
case $\sg_{2k}$ will also match $MMP(1,0,\emptyset,0)$ in $\sg$ for $1 < k < n$ but 
will not match  $MMP(1,0,\emptyset,0)$ in $\sg$ if 
$k=n$.  In addition, $\sg_1$ will match $MMP(1,0,\emptyset,0)$ in $\sg$ unless $\sg_1=2n$. 
Thus if $\sg = \sg_1 \ldots \sg_{2n} \in DU_{2n}$ 
and  $\mmp^{(1,0,\emptyset,0)}(\sg) =1$, then 
we must have $\sg_1 =2n$ and either $\sg_2 =1$ or 
$\sg_{2n} =1$. If $\sg_2 =1$, then none of 
$\sg_3, \ldots, \sg_{2n}$ will match  $MMP(1,0,\emptyset,0)$ in $\sg$. Clearly in such a situation, $\red[\sg_3 \ldots \sg_{2n}] \in DU_{2n-2}$ so that we have $C_{2n-2}(1) = A_{2n-2}(1)$ ways to choose 
$\sg_3 \ldots \sg_{2n}$. If $\sg_{2n} =1$, then 
it must be the case that $\sg_2 =2$ since otherwise $\sg_{2k} =2$ for 
some $k \geq 2$ in which case 
$\sg_2$ and $\sg_{2k} =2$ will match  $MMP(1,0,\emptyset,0)$ in $\sg$. But then $\red[\sg_3 \ldots \sg_{2n-1}] \in DU_{2n-3}$ so that we have $D_{2n-3}(1) = B_{2n-3}(1)$ ways to choose 
$\sg_3 \ldots \sg_{2n-1}$. Hence 
$C^{(1,0,\emptyset,0)}_{2n}(x)|_{x} = A_{2n-2}(1)+B_{2n-3}(1)$ for 
$n \geq 2$.

For (4), note that if $\sg = \sg_1 \ldots \sg_{2n+1} \in DU_{2n+1}$ where 
$n \geq 1$, then 
$\sg_{2}$ always matches $MMP(1,0,\emptyset,0)$ in $\sg$.  
Moreover if $\sg_2 \neq 1$, then $\sg_{2k} =1$ for some 
$k \geq 2$ in which case $\sg_{2k}$ will also match 
$MMP(1,0,\emptyset,0)$ in $\sg$. Finally, $\sg_1$ will also 
match $MMP(1,0,\emptyset,0)$ in $\sg$ unless $\sg_1 =2n+1$. 
Thus if $\sg = \sg_1 \ldots \sg_{2n+1} \in DU_{2n+1}$ 
and $\mmp^{(1,0,\emptyset,0)}(\sg) =1$, it must be 
the case that $\sg_1 =2n+1$ and $\sg_2 =1$ in which case none of 
$\sg_3, \ldots, \sg_{2n+1}$ will match  $MMP(1,0,\emptyset,0)$ in $\sg$. Clearly in such a situation, $\red[\sg_3 \ldots \sg_{2n}] \in DU_{2n-1}$ so that we have $D_{2n-1}(1) = B_{2n-1}(1)$ ways to choose 
$\sg_3 \ldots \sg_{2n+1}$. It follows that  
$D^{(1,0,\emptyset,0)}_{2n+1}(x)|_{x} = B_{2n-1}(1)$ for $n \geq 1$.
\end{proof}

We can also explain the coefficients of the highest power of 
$x$ in each of the polynomials $A_{2n}(x)$, $B_{2n+1}(x)$,  
and $D_{2n+1}(x)$.  That is, we have the following theorem.

\begin{theorem}\label{highest}\

\begin{itemize}
\item[\rm{(1)}] For all $n \geq 1$, the highest power of $x$ that appears in 
$A^{(1,0,\emptyset,0)}_{2n}(x)$ is $x^{n}$ which appears with coefficient 
$(2n-1)!!$. 

\item[\rm{(2)}]  For all $n \geq 1$, the highest power of $x$ that appears in 
$B^{(1,0,\emptyset,0)}_{2n+1}(x)$ is $x^{n}$ which appears with coefficient 
$(n+1)((2n-1)!!)$.

\item[\rm{(3)}]  For all $n \geq 2$, the highest power of $x$ that appears in 
$C^{(1,0,\emptyset,0)}_{2n}(x)$ is $x^{n}$ which appears with coefficient 
$(2n^2-n-1)((2n-4)!!)-n((2n-3)!!)$.

\item[\rm{(4)}]  For all $n \geq 2$, the highest power of $x$ that appears in 
$D^{(1,0,\emptyset,0)}_{2n-1}(x)$ is $x^{n+1}$ which appears with coefficient 
$(2n)!! -(2n-1)!!$.
\end{itemize}
\end{theorem}
\begin{proof}
For (1), we proceed by induction on $n$. Clearly the formula holds 
for $n=1$ since $A^{(1,0,\emptyset,0)}_2(x) = x$. Thus assume 
that $n > 1$ and that by induction, we know that 
$A^{(1,0,\emptyset,0)}_{2n-2}(x)|_{x^{n-1}} = (2n-3)!!$. 
It is easy to see  that the maximum that 
$\mmp^{(1,0,\emptyset,0)}(\sg)$ can be is $n$ since for 
any $\sg = \sg_1 \ldots \sg_{2n} \in UD_{2n}$ only 
$\sg_1,\sg_3, \ldots, \sg_{2n-1}$ can match $MMP(1,0,\emptyset,0)$ in $\sg$.  If $\sg_{2k+1} =1$ for some $k < n-1$, then 
$\sg_{2k+3}, \ldots, \sg_{2n-1}$ will not match 
$MMP(1,0,\emptyset,0)$ in $\sg$. Thus if 
$\sg = \sg_1 \ldots \sg_{2n} \in UD_{2n}$ is such that 
$\mmp^{(1,0,\emptyset,0)}(\sg) =n$, then 
$\sg_{2n-1} =1$ and $\mmp^{(1,0,\emptyset,0)}(\sg_1 \ldots \sg_{2n-2}) =n-1$.
We then  have $(2n-1)$ ways to choose the value $\sg_{2n}$ and, once 
we have chosen the value of $\sg_{2n}$, we have 
$(2n-3)!!$ ways to choose $\sg_1 \ldots \sg_{2n-2}$. 
Hence $A^{(1,0,\emptyset,0)}(x)|_{x^n} = (2n-1)!!$.

For (2), it is easy 
to see that our formula holds for $n=1$ and $n=2$ since 
$B^{(1,0,\emptyset,0)}_3(x)|_x =2$ and 
$B^{(1,0,\emptyset,0)}_5(x)|_{x^2} =9 = 3(3!!)$.
So assume that $n \geq 3$ and 
suppose that $\sg = \sg_1 \ldots \sg_{2n+1} \in UD_{2n+1}$. 
Then only $\sg_1,\sg_3, \ldots, \sg_{2n-1}$ can match the 
 $MMP(1,0,\emptyset,0)$ in $\sg$. 
Thus the maximum that $\mmp^{(1,0,\emptyset,0)}(\sg)$ can be is $n$. 
Note that if $\sg_{2k+1} =1$ where $0 \leq k < n-1$, then 
none of $\sg_{2j+1}$ for $j > k$ will match  the 
 $MMP(1,0,\emptyset,0)$ in $\sg$. It follows 
that if $\mmp^{(1,0,\emptyset,0)}(\sg) =n$, then it must be 
the case that $\sg_{2n+1}=1$ or $\sg_{2n-1}=1$. Now if 
$\sg_{2n-1}=1$, then we have $\binom{2n}{2}$ ways to choose 
the values of $\sg_{2n}$ and $\sg_{2n+1}$ and it 
must be the case that $\red[\sg_1 \ldots \sg_{2n-2}] =\tau$ where 
$\tau \in UD_{2n-2}$ and $\mmp^{(1,0,\emptyset,0)}(\tau) =n-1$.  
It then follows from part (1) that we have $(2n-3)!!$ ways to 
choose  $\sg_1 \ldots \sg_{2n-2}$ so 
that the set of permutations 
$\sg = \sg_1 \ldots \sg_{2n+1} \in UD_{2n+1}$ with $\sg_{2n-1} =1$ 
contributes $\binom{2n}{2}(2n-3)!! = n ((2n-1)!!)$ to 
$B^{(1,0,\emptyset,0)}_{2n+1}(x)|_{x^n}$.  If $\sg_{2n+1} = 1$, then 
it must be the case that $\sg_{2n-1}=2$ and 
$\red[\sg_1 \ldots \sg_{2n-2}] =\tau$ where 
$\tau \in UD_{2n-2}$ and $\mmp^{(1,0,\emptyset,0)}(\tau) =n-1$. Thus 
we have $2n-1$ choices for the value of $\sg_{2n}$ and then as before 
we have $(2n-3)!!$ to choose $\sg_1 \ldots \sg_{2n-2}$. 
Thus the set of permutations 
$\sg = \sg_1 \ldots \sg_{2n+1} \in UD_{2n+1}$ with $\sg_{2n+1} =1$ 
contributes $(2n-1)!!$ to 
$B^{(1,0,\emptyset,0)}_{2n+1}(x)|_{x^n}$. Hence 
 $B^{(1,0,\emptyset,0)}_{2n+1}(x)|_{x^n} = (n+1)(2n-1)!!$.

For (4), it is easy to see  that if $\sg = \sg_1 \ldots \sg_{2n+1} \in 
DU_{2n+1}$, then only $\sg_1,\sg_2, \sg_4, \ldots, \sg_{2n}$ can 
match  
$MMP(1,0,\emptyset,0)$ in $\sg$. Thus 
$\mmp^{(1,0,\emptyset,0)}(\sg)$ is at most $n+1$. It is also easy to see 
that if $\sg_{2k} = 1$ for $k <n$, then $\sg_{2k+2}, \ldots, \sg_{2n}$ will not match  
$MMP(1,0,\emptyset,0)$ in $\sg$ so that if 
$\mmp^{(1,0,\emptyset,0)}(\sg) =n+1$,
then it must be the case that $\sg_{2n} =1$. 
Thus assume that 
$\sg = \sg_1 \ldots \sg_{2n+1} \in 
DU_{2n+1}$ is such that $\mmp^{(1,0,\emptyset,0)}(\sg)=n+1$. 
Since $\sg_{2n}=1$, 
we have two cases.\\

\noindent
{\bf Case 1.} $\sg_{2n+1} = 2n+1$. In this case, we know that $\sg_1$ will always match 
 $MMP(1,0,\emptyset,0)$ in $\sg$. 
Thus as far as $\sg_1 \ldots \sg_{2n-1}$ is concerned,  
we are dealing with the polynomial 
$\overline{D}^{(1,0,\emptyset,0)}_{2n-1}(x)$. We then 
have the following lemma.

\begin{lemma}\label{BarDmax} For $n \geq 1$, the highest power of 
$x$ which occurs in $\overline{D}^{(1,0,\emptyset,0)}_{2n+1}(x)$ is $x^{n+1}$ 
which occurs 
with a coefficient of $(2n)!!$.
\end{lemma}  
\begin{proof}
We proceed by induction on $n$. Our theorem holds 
for $n=1$ since $\overline{D}^{(1,0,\emptyset,0)}_3(x) = 2x^2$. 
Now assume that $n > 1$ and the lemma holds for $n-1$.  As in our 
discussion for $D_{2n+1}(x)$, if $\sg = \sg_1 \ldots \sg_{2n+1} \in 
DU_{2n+1}$ is such that 
$\chi(\sg_1 =2n+1) + \mmp^{(1,0,\emptyset,0)}(\sg)=n+1$, 
then it must be the case that $\sg_{2n}=1$.  But then we have 
$2n$ choices for $\sg_{2n+1}$ and, once we have chosen 
$\sg_{2n+1}$, then $\tau = \red[\sg_1 \ldots \sg_{2n-1}]$ must 
be an element of $DU_{2n-1}$ such that  
$\chi(\tau_1 =2n-1) + \mmp^{(1,0,\emptyset,0)}(\tau)=n$. By induction, 
we have $(2(n-1))!!$ ways to pick $\sg_1 \ldots \sg_{2n-1}$. 
Thus  $\overline{D}^{(1,0,\emptyset,0)}_{2n+1}(x)|_{x^n} = (2n)!!$. 
\end{proof}

It follows that in Case 1, we have $(2n-2)!!$ ways to pick 
$\sg_1 \ldots \sg_{2n-1}$ so that the permutations 
such that $\sg_{2n}=1$ and $\sg_{2n+1} = 2n+1$ contribute 
$(2n-2)!!$ to $D^{(1,0,\emptyset,0)}_{2n+1}(x)|_{x^{n+1}}$.\\

\noindent {\bf Case 2.} $\sg_{2n+1} < 2n+1$.  In this case, $\tau = \red[\sg_1 \ldots \sg_{2n-1}]$ must 
be an element of $DU_{2n-1}$ such that  
$\mmp^{(1,0,\emptyset,0)}(\tau)=n$. It then follows 
by induction that we have 
$(2n-1)$ ways to pick $\sg_{2n+1}$ and, once we have chosen $\sg_{2n+1}$, we 
have $(2n-2)!!-(2n-3)!!$ ways to pick $\sg_1 \ldots \sg_{2n-1}$. 
Hence the permutations 
such that $\sg_{2n}=1$ and $\sg_{2n+1} < 2n+1$ contribute 
$(2n-1)((2n-2)!!-(2n-3)!!) = (2n-1)((2n-2)!!) - (2n-1)!!$ to $D^{(1,0,\emptyset,0)}_{2n+1}(x)|_{x^{n+1}}$.\\
\ \\
Thus 
\begin{eqnarray*}
D^{(1,0,\emptyset,0)}_{2n+1}(x)|_{x^{n+1}} &=& (2n-2)!! + (2n-1)((2n-2)!!) - (2n-1)!! \\
&=& (2n)!! - (2n-1)!!.
\end{eqnarray*}

For (3), observe that if $\sg = \sg_1 \ldots \sg_{2n} \in DU_{2n}$, 
then only $\sg_1, \sg_2, \sg_4, \ldots \sg_{2n-2}$ can 
match  $MMP(1,0,\emptyset,0)$ so that 
$\mmp^{(1,0,\emptyset,0)}(\sg)$ is at most $n$. It is also easy to see 
that if $\sg_{2k} = 1$ for $k <n-1$, then $\sg_{2k+2}, \ldots, \sg_{2n-2}$ will not match  
$MMP(1,0,\emptyset,0)$ in $\sg$ so that if $\mmp^{(1,0,\emptyset,0)}(\sg) =n$,
then it must be the case that $\sg_{2n} =1$ or $\sg_{2n-2} =1$. 
Suppose that $\sg = \sg_1 \ldots \sg_{2n} \in DU_{2n}$ and 
$\mmp^{(1,0,\emptyset,0)}(\sg) =n$. 
We then have three cases.\\

\noindent
{\bf Case I.} $\sg_{2n}=1$. In this case, it must be that 
$\tau = \red[\sg_1 \ldots \sg_{2n-1}] \in DU_{2n-1}$ and 
$\mmp^{(1,0,\emptyset,0)}(\tau) =n$. Thus by part (3), we 
have $(2n-2)!!-(2n-3)!!$ choices for $\sg_1 \ldots \sg_{2n}$.\\

\noindent
{\bf Case II.} $\sg_{2n-2} =1$ and $\sg_{2n-1} =2n$. In this case, we have $(2n-2)$ choices for $\sg_{2n}$. The fact 
that $\sg_{2n-1}=2n$ implies that $\sg_1$ will always match 
 $MMP(1,0,\emptyset,0)$ so 
that $\gamma = \red[\sg_1 \ldots \sg_{2n-3}]$ is a permutation 
in $DU_{2n-3}$ such that 
$\chi(\sg_1 =2n-3) + \mmp^{(1,0,\emptyset,0)}(\gamma)=n-1$. 
By Lemma \ref{BarDmax}, we will have $(2n-4)!!$ choices 
for $\sg_1 \ldots \sg_{2n-3}$ once we have chosen $\sg_{2n}$. 
Thus the permutations in Case II will contribute 
$(2n-2)!!$ to $C^{(1,0,\emptyset,0)}_{2n}(x)|_{x^n}$. \\

\noindent
{\bf Case III.} $\sg_{2n-2} =1$ and $\sg_{2n-1} < 2n$. In this case, $\tau = \red[\sg_1 \ldots \sg_{2n-3}]$ must 
be an element of $DU_{2n-3}$ such that  
$\mmp^{(1,0,\emptyset,0)}(\tau)=n-1$. Then we have 
$\binom{2n-2}{2}$ ways to pick $\sg_{2n-1}$ and $\sg_{2n}$ and once we have chosen $\sg_{2n-1}$ and $\sg_{2n}$, we 
have $(2n-4)!!-(2n-5)!!$ ways to pick $\sg_1 \ldots \sg_{2n-3}$ by part (4). 
It follows that the permutations in Case III
contribute 
$\binom{2n-2}{2}((2n-4)!!-(2n-5)!!)$ to 
$C^{(1,0,\emptyset,0)}_{2n}(x)|_{x^n}$.\\
\ \\
Thus 
\begin{eqnarray*} 
C^{(1,0,\emptyset,0)}_{2n}(x)|_{x^n} &=& (2n-2)!!-(2n-3)!! + (2n-2)!!+ 
\binom{2n-2}{2} ((2n-4)!!-(2n-5)!!) \\
&=& 2((2n-2)!!)-(2n-3)!! + (n-1)(2n-3)((2n-4)!!-(2n-5)!!) \\
&=& (2(2n-2)+(n-1)(2n-3))((2n-4)!!) -n((2n-3)!!) \\
&=& (2n^2 -n-1)((2n-4)!!)  -n((2n-3)!!).
\end{eqnarray*}

\end{proof}

Next we give formulas for the coefficient of 
$x^2$ in the polynomials $A^{(1,0,\emptyset,0)}_{2n}(x)$, 
$B^{(1,0,\emptyset,0)}_{2n+1}(x)$, $C^{(1,0,\emptyset,0)}_{2n}(x)$, 
and $B^{(1,0,\emptyset,0)}_{2n+1}(x)$. None of the corresponding 
sequences had previously appeared in the OEIS \cite{oeis}. 

\begin{theorem} \

\begin{itemize}
\item[\rm{(1)}] For $n \geq 2$,
\begin{equation*} 
A^{(1,0,\emptyset,0)}_{2n}(x)|_{x^2} = 
\sum_{k=1}^{n-1} \binom{2n-1}{2k} B_{2k-1}(1) B_{2n-2k-1}(1).
\end{equation*}
\item[\rm{(2)}]  For $n \geq 3$,
\begin{equation*} 
B^{(1,0,\emptyset,0)}_{2n+1}(x)|_{x^2} = A^{(1,0,\emptyset,0)}_{2n}(x)|_{x^2}
+ \sum_{k=1}^{n-1} \binom{2n}{2k} B_{2k-1}(1) A_{2n-2k}(1).
\end{equation*}
\item[\rm{(3)}]  For $n \geq 2$,
\begin{equation*} 
D^{(1,0,\emptyset,0)}_{2n+1}(x)|_{x^2} = (2n-1)B_{2n-1}(1) + 
\sum_{k=2}^{n-1} \binom{2n-1}{2k-2} B_{2k-3}(1) B_{2n-2k+1}(1).
\end{equation*}
\item[\rm{(4)}]  For $n \geq 2$, 
\begin{eqnarray*} 
C^{(1,0,\emptyset,0)}_{2n}(x)|_{x^2} &=& D^{(1,0,\emptyset,0)}_{2n-1}(x)|_{x^2}
+(2n-2)A_{2n-2}(1) + \\
&&
\sum_{k=2}^{n-1} \binom{2n-2}{2k-2} B_{2k-3}(1) A_{2n-2k}(1). \nonumber 
\end{eqnarray*}
\end{itemize}
\end{theorem}
\begin{proof}

For (1), suppose that $\sg = \sg_1 \ldots \sg_{2n} \in UD_{2n}$ and 
$\mmp^{(1,0,\emptyset,0)}(\sg) =2$. Then it cannot be 
that $\sg_1=1$ since that would force that $\sg_2, \ldots, \sg_{2n}$ 
do not match  $MMP(1,0,\emptyset,0)$ in $\sg$. 
Thus $1 \in \{\sg_{2k+1}: k =1, \ldots ,n-1\}$. 
Now suppose that  
$\sg_{2k+1} =1$ where $1 \leq k \leq n-1$. Then $\sg_{2k+1}$ 
will match  $MMP(1,0,\emptyset,0)$ in $\sg$ and 
$\sg_{2k+2}, \ldots, \sg_{2n}$ will 
not match  $MMP(1,0,\emptyset,0)$ in $\sg$. 
Hence it must be the case that 
$\tau = \red[\sg_1 \ldots \sg_{2k}]$ is a permutation in $UD_{2k}$ 
such that $\mmp^{(1,0,\emptyset,0)}(\tau) =1$.  Thus we have 
$\binom{2n-1}{2k}$ ways to choose the set of elements for 
$\sg_1, \ldots, \sg_{2k}$ and, by Theorem \ref{lowest}, we 
have $B_{2k-1}(1)$ ways to order them.  We also 
have $B_{2n-2k-1}(1)$ ways to order $\sg_{2k+2} \ldots \sg_{2n}$. 
Hence 
$$A^{(1,0,\emptyset,0)}_{2n}(x)|_{x^2} = 
\sum_{k=1}^{n-1} \binom{2n-1}{2k} B_{2k-1}(1) B_{2n-2k-1}(1).$$

The argument for (2) is similar.  That is, suppose $\sg = \sg_1 \ldots \sg_{2n+1} \in UD_{2n+1}$ and $\mmp^{(1,0,\emptyset,0)}(\sg) =2$. Then 
again we cannot have $\sg_1 =1$.  
Thus $1 \in \{\sg_{2k+1}: k =1, \ldots ,n\}$. 
Now suppose 
$\sg_{2k+1} =1$ where $1 \leq k \leq n-1$. Then $\sg_{2k+1}$ 
will match  $MMP(1,0,\emptyset,0)$ in $\sg$ and 
$\sg_{2k+2}, \ldots, \sg_{2n+1}$ will 
not match  $MMP(1,0,\emptyset,0)$ in $\sg$. 
Hence it must be the case that 
$\tau = \red[\sg_1 \ldots \sg_{2k}]$ is a permutation in $UD_{2k}$ 
such that $\mmp^{(1,0,\emptyset,0)}(\tau) =1$.  Thus we have 
$\binom{2n}{2k}$ ways to choose the set of elements for 
$\sg_1, \ldots, \sg_{2k}$ and, by Theorem \ref{lowest}, we 
have $B_{2k-1}(1)$ ways to order them.  We also 
have $A_{2n-2k}(1)$ ways to order $\sg_{2k+2} \ldots \sg_{2n+1}$. 
However if $\sg_{2n+1} = 1$, then $\sg_{2n+1}$ does 
not match  $MMP(1,0,\emptyset,0)$ in $\sg$ so 
that it must be the case that 
$\alpha = \red[\sg_1 \ldots \sg_{2n}]$ is an element of 
$UD_{2n}$ such that  $\mmp^{(1,0,\emptyset,0)}(\tau) =2$. It follows 
that in this case, we have 
$A^{(1,0,\emptyset,0)}_{2n}(x)|_{x^2}$ ways to choose 
$\sg_1 \ldots \sg_{2n}$. 
Hence 
$$B^{(1,0,\emptyset,0)}_{2n+1}(x)|_{x^2} = 
A^{(1,0,\emptyset,0)}_{2n}(x)|_{x^2} +
\sum_{k=1}^{n-1} \binom{2n}{2k} B_{2k-1}(1) A_{2n-2k}(1).$$

For part (3), suppose that 
$\sg = \sg_1 \ldots \sg_{2n+1} \in DU_{2n+1}$ and 
$\mmp^{(1,0,\emptyset,0)}(\sg) =2$.  Then 
$1 \in \{\sg_2,\sg_4, \ldots, \sg_{2n}\}$.  Now if 
$\sg_2 =1$, then we cannot have 
$\sg_1 = 2n+1$ because that would force 
$\mmp^{(1,0,\emptyset,0)}(\sg) =1$. Thus if 
$\sg_2 =1$, then $2\leq \sg_1 \leq 2n$ in which 
case $\sg_1$ and $\sg_2$ will be the only two elements 
of $\sg$ to match  
 $MMP(1,0,\emptyset,0)$ in $\sg$.
We then have $D_{2n-1}(1) = B_{2n-1}(1)$ ways to 
pick $\sg_3 \ldots \sg_{2n+1}$ as 
$\red[\sg_3 \ldots \sg_{2n+1}] \in DU_{2n-1}$. Thus the 
number of $\sg = \sg_1 \ldots \sg_{2n+1} \in DU_{2n+1}$ such that $\mmp^{(1,0,\emptyset,0)}(\sg) =2$ and $\sg_2 =1$ is 
$(2n-1)B_{2n-1}(1)$. Next assume that 
$\sg_{2k} =1$ where $2 \leq k \leq n$. Then $\sg_{2k}$ matches   
 $MMP(1,0,\emptyset,0)$ in $\sg$. It follows 
that we cannot have $2n+1 \in \{\sg_3, \ldots \sg_{2n+1}\}$ since 
otherwise $\sg_1$ and $\sg_2$ would also match  
 $MMP(1,0,\emptyset,0)$ in $\sg$ which would 
force $\mmp^{(1,0,\emptyset,0)}(\sg) \geq 3$. Thus it 
must be the case that $\sg_1 =2n+1$. Moreover, 
if $s = \min(\{\sg_2, \ldots, \sg_{2k-1}\})$, then it 
must be the case that $\sg_2 =s$ since otherwise 
$s = \sg_{2j}$ for some $2 \leq j \leq k-1$ in which 
case both $\sg_2$ and $\sg_{2j}$ would  match  
 $MMP(1,0,\emptyset,0)$ in $\sg$ which would 
mean $\mmp^{(1,0,\emptyset,0)}(\sg) \geq 3$. 
Thus  
we have $\binom{2n-1}{2k-2}$ ways to choose 
the elements $\sg_3, \ldots, \sg_{2k-1}$ and 
then we have $B_{2k-3}(1)$ ways to order 
$\sg_3, \ldots, \sg_{2k-1}$ since 
$\red[\sg_3 \ldots \sg_{2k-1}]$ must be an element 
of $DU_{2k-3}$ and we have $B_{2n-2k+1}(1)$ ways to order 
$\sg_{2k+1} \ldots \sg_{2n+1}$ since 
$\red[\sg_{2k+1} \ldots \sg_{2n+1}]$ must be an element 
of $DU_{2n-2k+1}$. Hence, 
$$D^{(1,0,\emptyset,0)}_{2n+1}(x)|_{x^2} = (2n-1)B_{2n-1}(1) + 
\sum_{k=2}^{n-1} \binom{2n-1}{2k-2} B_{2k-3}(1) B_{2n-2k+1}(1).
$$

For part (4), suppose that 
$\sg = \sg_1 \ldots \sg_{2n} \in DU_{2n}$ and 
$\mmp^{(1,0,\emptyset,0)}(\sg) =2$.  Then 
$1 \in \{\sg_2,\sg_4, \ldots, \sg_{2n}\}$.  We then 
have three cases. \\

\noindent
{\bf Case 1.} $\sg_{2n} =1$. In this case, $\sg_{2n}$ does not match  $MMP(1,0,\emptyset,0)$ in $\sg$ so that it must 
be the case that if $\tau = \red[\sg_1 \ldots \sg_{2n-1}]$, then 
$\tau$ is a permutation in $DU_{2n-1}$ such that 
$\mmp^{(1,0,\emptyset,0)}(\tau)=2$. Thus, by part (3), we 
have $D^{(1,0,\emptyset,0)}_{2n-1}(x)|_{x^2}$ choices 
for $\sg_1 \ldots \sg_{2n-1}$.\\

\noindent
{\bf Case 2.} $\sg_2 =1$. In this case, we cannot have 
$\sg_1 = 2n$ because that would force 
$\mmp^{(1,0,\emptyset,0)}(\sg) =1$. Thus if 
$\sg_2 =1$, then $2\leq \sg_1 \leq 2n-1$ in which 
case $\sg_1$ and $\sg_2$ will be the only two elements 
of $\sg$ to match  
 $MMP(1,0,\emptyset,0)$ in $\sg$.
We then have $C_{2n-2}(1) = A_{2n-2}(1)$ ways to 
pick $\sg_3 \ldots \sg_{2n}$ as 
$\red[\sg_3 \ldots \sg_{2n}] \in DU_{2n-2}$. Thus the 
number of $\sg = \sg_1 \ldots \sg_{2n} \in DU_{2n}$ such 
that $\mmp^{(1,0,\emptyset,0)}(\sg) =2$ and $\sg_2 =1$ is 
$(2n-2)A_{2n-2}(1)$. \\

\noindent
{\bf Case 3.} $\sg_{2k} =1$ where $2 \leq k \leq n-1$. Then $\sg_{2k}$ 
matches  
 $MMP(1,0,\emptyset,0)$ in $\sg$. It follows 
that we cannot have $2n+1 \in \{\sg_3, \ldots \sg_{2n}\}$ since 
otherwise $\sg_1$ and $\sg_2$ would also match  
 $MMP(1,0,\emptyset,0)$ in $\sg$ which would 
force $\mmp^{(1,0,\emptyset,0)}(\sg) \geq 3$. Thus it 
must be the case that $\sg_1 =2n$. Again, 
if $s = \min(\{\sg_2, \ldots, \sg_{2k-1}\})$, then it 
must be the case that $\sg_2 =s$ since otherwise 
$s = \sg_{2j}$ for some $2 \leq j \leq k-1$ in which 
case both $\sg_2$ and $\sg_{2j}$ would match  
 $MMP(1,0,\emptyset,0)$ in $\sg$ which would 
force  $\mmp^{(1,0,\emptyset,0)}(\sg) \geq 3$. 
It follows that 
we have $\binom{2n-3}{2k-2}$ ways to choose 
the elements $\sg_2, \ldots, \sg_{2k-1}$ and 
then we have $B_{2k-3}(1)$ ways to order 
$\sg_3, \ldots, \sg_{2k-1}$ since 
$\red[\sg_3 \ldots \sg_{2k-1}]$ must be an element 
of $DU_{2n-3}$ and $A_{2n-2k}(1)$ ways to order 
$\sg_{2k+1} \ldots \sg_{2n}$ since 
$\red[\sg_{2k+1} \ldots \sg_{2n}]$ must be an element 
of $DU_{2n-2k}$. Hence the elements in Case 3 contribute 
$\sum_{k=2}^{n-1} \binom{2n-3}{2k-2} B_{2k-3}(1) A_{2n-2k}(1)$ to 
$C^{(1,0,\emptyset,0)}_{2n}(x)|_{x^2}$. Hence 
\begin{eqnarray*}
C^{(1,0,\emptyset,0)}_{2n}(x)|_{x^2} &=& 
D^{(1,0,\emptyset,0)}_{2n-1}(x)|_{x^2}+ (2n-2)A_{2n-2}(1) +\\
&&\sum_{k=2}^{n-1} \binom{2n-2}{2k-2} B_{2k-3}(1) A_{2n-2k}(1).
\end{eqnarray*}
\end{proof}

Finally, we have the following theorem which gives formulas 
for the second highest coefficient in 
$A^{(1,0,\emptyset,0)}_{2n}(x)$, $B^{(1,0,\emptyset,0)}_{2n+1}(x)$, 
$C^{(1,0,\emptyset,0)}_{2n}(x)$, and $D^{(1,0,\emptyset,0)}_{2n+1}(x)$. 
 None of the corresponding 
sequences had previously appeared in the OEIS \cite{oeis}.

\begin{theorem} \

\begin{itemize} 
 \item[\rm{(1)}]  For all $n \geq 2$, 
\begin{equation}\label{Asectop} 
A^{(1,0,\emptyset,0)}_{2n}(x)|_{x^{n-1}} = \frac{2}{3}\binom{n}{2} 
((2n-1)!!).
\end{equation}
\item[\rm{(2)}]  For all $n \geq 2$, 
\begin{equation}\label{Bsectop} 
B^{(1,0,\emptyset,0)}_{2n+1}(x)|_{x^{n-1}} = 
\left(\frac{7}{3}\binom{n}{2} + 2\binom{n}{3}\right)((2n-1)!!).
\end{equation}

\item[\rm{(3)}]  For all $n \geq 1$, 
\begin{eqnarray}\label{Dsectop} 
D^{(1,0,\emptyset,0)}_{2n+1}(x)|_{x^{n}} &=& 
\left(\sum_{k=1}^n \frac{(5k-4)k}{3} ((2k-2)!!)\prod_{i=k+1}^n (2i-1)\right) 
- \\ 
&&\frac{2}{3}\left(\binom{n}{2}-1\right)((2n-1)!!). \nonumber 
\end{eqnarray}  

\item[\rm{(4)}]  For all $n \geq 3$, 
\begin{eqnarray} \label{Csectop}
C^{(1,0,\emptyset,0)}_{2n}(x)|_{x^{n-1}} &=&
D^{(1,0,\emptyset,0)}_{2n-1}(x)|_{x^{n-1}} + \binom{2n-2}{2}
D^{(1,0,\emptyset,0)}_{2n-3}(x)|_{x^{n-2}} + \\
&&  \frac{28n^2-72n+39}{24}((2n-2)!!) - \frac{5}{3}\binom{n-1}{2}((2n-3)!!). 
\nonumber 
\end{eqnarray}
\end{itemize}
\end{theorem}
\begin{proof}

For (1), we proceed by induction on $n$. Now (\ref{Asectop}) holds 
for $n=2$ since $A^{(1,0,\emptyset,0)}_{4}(x)|_{x}=2$. 
Now suppose that $n > 2$, $\sg = \sg_1 \ldots \sg_{2n} \in UD_{2n}$   
and $\mmp^{(1,0,\emptyset,0)}(\sg) = n-1$.  Only $\sg_1, \sg_3, \ldots, 
\sg_{2n-1}$ can match  $MMP(1,0,\emptyset,0)$ in 
$\sg$.  Now it cannot be that $\sg_{2k+1} =1$ where $k < n-2$ since 
then $\mmp^{(1,0,\emptyset,0)}(\sg) \leq k+1< n-1$. Thus it 
must be the case that $\sg_{2n-3} =1$ or $\sg_{2n-1}=1$. 
Now if $\sg_{2n-3} =1$, then $\sg_{2n-3}$ matches  $MMP(1,0,\emptyset,0)$ in 
$\sg$ and $\sg_{2n-1}$ does not match
  $MMP(1,0,\emptyset,0)$ in $\sg$. Then we have 
$\binom{2n-1}{3}$ ways to 
choose the values of $\sg_{2n-2}$, $\sg_{2n-1}$, and $\sg_{2n}$ and 
we have two ways to order them. In addition, we must have that 
$ \mmp^{(1,0,\emptyset,0)}(\red[\sg_1 \ldots, \sg_{2n-4}]) = n-2$. 
But then by Theorem \ref{highest}, we have 
$(2n-5)!!$ ways to choose $\sg_1 \ldots \sg_{2n-4}$ so that the 
number of $\sg = \sg_1 \ldots \sg_{2n} \in UD_{2n}$ such that   
$\sg_{2n-3}=1$ and $\mmp^{(1,0,\emptyset,0)}(\sg) = n-1$ is 
\begin{equation*}
2\binom{2n-1}{3}(2n-5)!! = \frac{2}{3}(n-1)((2n-1)!!).
\end{equation*}
Now if $\sg_{2n-1}=1$, then we have $2n-1$ ways to pick the 
value of $\sg_{2n}$ and we must have that 
$ \mmp^{(1,0,\emptyset,0)}(\red[\sg_1 \ldots, \sg_{2n-2}]) = n-2$. 
Thus we have $A_{2(n-1)}(x)|_{x^{n-2}} = \frac{2}{3}\binom{n-1}{2}(2n-3)!!$ 
ways to choose $\sg_1 \ldots \sg_{2n-2}$. Thus the 
number of $\sg = \sg_1 \ldots \sg_{2n} \in UD_{2n}$ such that   
$\sg_{2n-1}=1$ and $\mmp^{(1,0,\emptyset,0)}(\sg) = n-1$ is 
$\frac{2}{3}\binom{n-1}{2}(2n-1)!!$. Hence 
$$
A_{2n}(x)|_{x^{n-1}}= \frac{2}{3}(n-1)((2n-1)!!) + 
\frac{2}{3}\binom{n-1}{2}((2n-1)!!) = \frac{2}{3}\binom{n}{2}((2n-1)!!).
$$

Part (2) can be proved by induction in a similar manner. 
Now (\ref{Bsectop}) holds 
for $n=2$ since $B^{(1,0,\emptyset,0)}_{5}(x)|_{x}=7$. 
Now suppose that $n > 2$, $\sg = \sg_1 \ldots \sg_{2n+1} \in UD_{2n+1}$,  
and $\mmp^{(1,0,\emptyset,0)}(\sg) = n-1$.  Only $\sg_1, \sg_3, \ldots, 
\sg_{2n-1}$ can match  $MMP(1,0,\emptyset,0)$ in 
$\sg$.  Again it cannot be that $\sg_{2k+1} =1$ where $k < n-2$ since 
then $\mmp^{(1,0,\emptyset,0)}(\sg) \leq k+1 < n-1$. Thus it 
must be the case that $\sg_{2n-3} =1$, $\sg_{2n-1}=1$, 
and $\sg_{2n+1}=1$. Thus we have three cases. \\

\noindent
{\bf Case A.} $\sg_{2n-3} =1$. Then $\sg_{2n-3}$ matches  $MMP(1,0,\emptyset,0)$ in 
$\sg$ and $\sg_{2n-1}$ does not match
  $MMP(1,0,\emptyset,0)$ in $\sg$. Then we have 
$\binom{2n}{4}$ ways to 
choose the values of $\sg_{2n-2}$, $\sg_{2n-1}$, $\sg_{2n}$, and 
$\sg_{2n+1}$, and 
we have 5 ways to order them. In addition, we must have that 
$ \mmp^{(1,0,\emptyset,0)}(\red[\sg_1 \ldots, \sg_{2n-4}]) = n-2$. 
But then by Theorem \ref{highest}, we have 
$(2n-5)!!$ ways to pick $\sg_1 \ldots \sg_{2n-4}$ so that the 
number of $\sg = \sg_1 \ldots \sg_{2n} \in UD_{2n}$ such that   
$\sg_{2n-3}=1$ and $\mmp^{(1,0,\emptyset,0)}(\sg) = n-1$ is 
\begin{equation*}
5\binom{2n}{4}(2n-5)!! = \frac{5}{6}n(n-1)((2n-1)!!) = 
\frac{5}{3}\binom{n}{2}((2n-1)!!).
\end{equation*}

\noindent
{\bf Case B.}  $\sg_{2n-1}=1$. 
Then we have $\binom{2n}{2}$ ways to pick the 
values of $\sg_{2n}$ and $\sg_{2n+1}$ and we must have that 
$ \mmp^{(1,0,\emptyset,0)}(\red[\sg_1 \ldots, \sg_{2n-2}]) = n-2$. 
Thus we have $A_{2(n-1)}(x)|_{x^{n-2}} = \frac{2}{3}\binom{n-1}{2}(2n-3)!!)$ 
ways to pick $\sg_1 \ldots \sg_{2n-2}$. 
Thus 
number of $\sg = \sg_1 \ldots \sg_{2n} \in UD_{2n}$ such that   
$\sg_{2n-1}=1$ and $\mmp^{(1,0,\emptyset,0)}(\sg) = n-1$ is 
$$\binom{2n}{2} \frac{2}{3}\binom{n-1}{2}((2n-3)!!)=   
\frac{2}{3}n\binom{n-1}{2}((2n-1)!!) = 2\binom{n}{3}((2n-1)!!).$$

\noindent
{\bf Case C.} $\sg_{2n+1} =1$. In this case, $\sg_{2n+1}$ does not match 
$MMP(1,0,\emptyset,0)$ in $\sg$ so that we must have 
that $ \mmp^{(1,0,\emptyset,0)}(\red[\sg_1 \ldots, \sg_{2n}]) = n-1$. 
By part (1), we have $\frac{2}{3} \binom{n}{2} ((2n-1)!!)$ ways to choose 
$\sg_1 \ldots \sg_{2n}$ in this case. \\
\ \\
Thus it follows that  
$B_{2n+1}(x)|_{x^{n-1}}= (\frac{7}{3}\binom{n}{2} + 2\binom{n}{3})((2n-1)!!).$

Before we can prove part (3), we first need to establish the 
following lemma.

\begin{lemma}\label{barDsectop}  For $n \geq 1$, 
\begin{equation}
\overline{D}_{2n+1}(x)|_{x^n} = \frac{1}{3}(n^2-1)((2n)!!).
\end{equation}
\end{lemma}
\begin{proof}
We proceed by induction on $n$.  The lemma holds 
for $n=1$ since $\overline{D}^{(1,0,\emptyset,0)}_3(x) = 2x^2$ so 
that $\overline{D}_3(x)|_{x}=0$. 
Now assume that $n \geq 2$ and 
$\sg = \sg_1 \ldots \sg_{2n+1} \in DU_{2n+1}$,  
and $\mmp^{(1,0,\emptyset,0)}(\sg) + \chi(\sg_1 = 2n+1) = n$.  Only $\sg_1,\sg_2, \sg_4, \ldots, 
\sg_{2n}$ can match  $MMP(1,0,\emptyset,0)$ in 
$\sg$.  
Now it cannot be that $\sg_{2k} =1$ where $k \leq  n-2$ since 
then $\mmp^{(1,0,\emptyset,0)}(\sg) + \chi(\sg_1 = 2n+1) \leq k+1 < n$. Thus it 
must be the case that $\sg_{2n-2} =1$ or $\sg_{2n}=1$. 
Now if $\sg_{2n-2} =1$, then $\sg_{2n-2}$ matches  $MMP(1,0,\emptyset,0)$ in 
$\sg$ and $\sg_{2n}$ does not match
  $MMP(1,0,\emptyset,0)$ in $\sg$. Then we have 
$\binom{2n}{3}$ ways to 
choose the values of $\sg_{2n-1}$, $\sg_{2n}$, and $\sg_{2n+1}$ and 
we have two ways to order them. In addition, we must have that if 
$\tau = \red[\sg_1 \ldots, \sg_{2n-3}]$, then 
$ \mmp^{(1,0,\emptyset,0)}(\tau) + \chi(\tau_1 = 2n-3)= n-1$. 
By Lemma \ref{BarDmax}, 
we then have 
$(2n-4)!!$ ways to choose $\sg_1 \ldots \sg_{2n-3}$ so that the 
number of $\sg = \sg_1 \ldots \sg_{2n+1} \in DU_{2n+1}$ such that   
$\sg_{2n-2}=1$ and $\mmp^{(1,0,\emptyset,0)}(\sg) + \chi(\sg_1 = 2n+1)= n$ is 
\begin{equation*}
2\binom{2n}{3}(2n-4)!! = \frac{2n-1}{3}((2n)!!).
\end{equation*}

Now if $\sg_{2n}=1$, then we have $2n$ ways to pick the 
value of $\sg_{2n+1}$ and if \\
$\alpha = \red[\sg_1 \ldots, \sg_{2n-1}]$, then 
$ \mmp^{(1,0,\emptyset,0)}(\alpha) + \chi(\alpha_1 = 2n-1)= n-1$. 
Then we have $\overline{D}_{2(n-1)+1}(x)|_{x^{n-1}} = \frac{1}{3}((n-1)^2-1)
((2n-2)!!)$ 
ways to choose $\sg_1 \ldots \sg_{2n-1}$. Thus 
number of $\sg = \sg_1 \ldots \sg_{2n} \in UD_{2n}$ such that   
$\sg_{2n-1}=1$ and $\mmp^{(1,0,\emptyset,0)}(\sg) + \chi(\sg_1 = 2n+1)= n$ is 
$\frac{1}{3}((n-1)^2-1)(2n)!!$. Hence 
$$\overline{D}_{2n+1}(x)|_{x^{n}}= \frac{2n-1}{3}((2n)!!) + \frac{1}{3}
((n-1)^2-1)((2n)!!) = \frac{1}{3}(n^2-1)((2n)!!).$$
\end{proof}

We prove part (3) by induction. 
We have that (\ref{Dsectop}) holds 
for $n=2$ since $D^{(1,0,\emptyset,0)}_{5}(x)|_{x}=9$. 
Now suppose that $n > 2$, $\sg = \sg_1 \ldots \sg_{2n+1} \in DU_{2n+1}$,  
and $\mmp^{(1,0,\emptyset,0)}(\sg) = n$.  Only $\sg_1,\sg_2, \sg_4, \ldots, 
\sg_{2n}$ can match $MMP(1,0,\emptyset,0)$ in 
$\sg$.  It cannot be that $\sg_{2k} =1$ where $k \leq  n-2$ since 
then $\mmp^{(1,0,\emptyset,0)}(\sg) \leq k+1<n$. Hence it 
must be the case that $\sg_{2n-2} =1$ or $\sg_{2n}=1$. 
Thus we have two cases. \\

\noindent
{\bf Case I.} $\sg_{2n-2} =1$.  Then $\sg_{2n-2}$ matches  $MMP(1,0,\emptyset,0)$ in 
$\sg$ and $\sg_{2n}$ does not match
  $MMP(1,0,\emptyset,0)$ in $\sg$. Then 
we have two subcases.\\

\noindent
{\bf Subcase I.a.} $2n+1 \in \{\sg_{2n-1},\sg_{2n},\sg_{2n+1}\}$. In this case, we have $\binom{2n-1}{2}$ ways to choose the values 
of the other 2 elements in  the set $\{\sg_{2n-1},\sg_{2n},\sg_{2n+1}\}$ and 
then we have 
2 ways to order $\sg_{2n-1}\sg_{2n}\sg_{2n+1}$. Then since 
 $2n+1 \in \{\sg_{2n-1},\sg_{2n},\sg_{2n+1}\}$, 
we are guaranteed that $\sg_1$ matches  
$MMP(1,0,\emptyset,0)$  in $\sg$.  Thus when 
we consider $\tau =\red[\sg_1 \ldots \sg_{2n-3}]$, we must have that  
$\mmp^{(1,0,\emptyset,0)}(\tau)+\chi(\tau_1 = 2n-3)= n-1$. It follows 
from Lemma \ref{BarDmax} that we have 
$(2n-4)!!$ ways to pick $\sg_1 \ldots \sg_{2n-3}$.  Hence the 
$\sg \in DU_{2n+1}$  in this subcase contribute 
$2\binom{2n-1}{2}(2n-4)!! = (2n-1)((2n-2)!!)$ to 
$D^{(1,0,\emptyset,0)}_{2n+1}(x)|_{x^{n}}$. \\

\noindent
{\bf Subcase I.b.} $2n+1 \not\in \{\sg_{2n-1},\sg_{2n},\sg_{2n+1}\}$. We then have $\binom{2n-1}{3}$ ways to choose the values 
of the elements of $\{\sg_{2n-1},\sg_{2n},\sg_{2n+1}\}$ and 
2 ways to order them. Because $2n+1 \not \in \{\sg_{2n-1},\sg_{2n},\sg_{2n+1}\}$, we are not guaranteed that $\sg_1$ matches  
$MMP(1,0,\emptyset,0)$  in $\sg$.  Thus when we consider $\tau =\red[\sg_1 \ldots \sg_{2n-3}]$, we must have that 
$\mmp^{(1,0,\emptyset,0)}(\tau)= n-1$. It follows 
from Theorem \ref{highest}  that we have 
$(2n-4)!!-(2n-5)!!$ ways to pick $\sg_1 \ldots \sg_{2n-3}$. 
Thus the $\sg \in DU_{2n+1}$  in this subcase contribute 
$$2\binom{2n-1}{3}((2n-4)!!-(2n-5)!!) =
\frac{(2n-1)(2n-3)}{3} ((2n-2)!!) -\frac{2}{3}(n-1)((2n-1)!!)$$
to 
$D^{(1,0,\emptyset,0)}_{2n+1}(x)|_{x^{n}}$. \\

\noindent
{\bf Case II.}  $\sg_{2n}=1$. In this case $\sg_{2n}$ matches  
$MMP(1,0,\emptyset,0)$ in $\sg$. Then again, we have two subcases.\\

\noindent
{\bf Subcase II.a.} $\sg_{2n+1} = 2n+1$. Because $2n+1 = \sg_{2n+1}$, 
we are guaranteed that $\sg_1$ matches  
$MMP(1,0,\emptyset,0)$ in $\sg$.  Thus when we consider $\tau =\red[\sg_1 \ldots \sg_{2n-1}]$, we must have that 
$\mmp^{(1,0,\emptyset,0)}(\tau)+\chi(\tau_1=2n-1)= n-1$. It follows 
from Lemma \ref{barDsectop} that we have 
$\frac{1}{3}((n-1)^2-1)((2n-2)!!)$ ways to choose $\sg_1 \ldots \sg_{2n-3}$.  Hence the 
$\sg \in DU_{2n+1}$  in this subcase contribute 
$\frac{1}{3}((n-1)^2-1)((2n-2)!!)$ to 
$D^{(1,0,\emptyset,0)}_{2n+1}(x)|_{x^{n}}$. \\

\noindent
{\bf Subcase II.b.} $2n+1 \neq \sg_{2n+1}$. We then have $(2n-1)$ ways to choose the value  
of $\sg_{2n+1}$. Because $2n+1 \neq \sg_{2n+1}$, then 
we are not guaranteed that $\sg_1$ matches  
$MMP(1,0,\emptyset,0)$ in $\sg$.  Thus when we consider 
$\tau =\red[\sg_1 \ldots \sg_{2n-1}]$, we must have that 
$\mmp^{(1,0,\emptyset,0)}(\tau)= n-1$.  It then follows by induction 
that the permutations $\sg \in DU_{2n+1}$ in this subcase 
contribute 
\begin{eqnarray*}
(2n-1)\left(\left(\sum_{k=1}^{n-1} \frac{(5k-4)k}{3} ((2k-2)!!)\prod_{i=k+1}^{n-1} (2i-1)\right) 
- \frac{2}{3}\left(\binom{n-1}{2}-1\right)((2n-3)!!)\right) = \\
\left(\sum_{k=1}^{n-1} \frac{(5k-4)k}{3} ((2k-2)!!)\prod_{i=k+1}^{n} (2i-1)\right) 
- \frac{2}{3}\left(\binom{n-1}{2}-1\right)((2n-1)!!)
\end{eqnarray*}
to 
$D^{(1,0,\emptyset,0)}_{2n+1}(x)|_{x^{n}}$.\\
\ \\ 
It follows that 

\begin{eqnarray*}
D^{(1,0,\emptyset,0)}_{2n+1}(x)|_{x^n} &=&  
(2n-1)((2n-2)!!) + \frac{1}{3}((n-1)^2-1)((2n-2)!!) + \\
&&\frac{(2n-1)(2n-3)}{3} ((2n-2)!!) -
\frac{2n-2}{3}((2n-1)!!)+\\ 
&&\left(\sum_{k=1}^{n-1} \frac{(5k-4)k}{3} ((2k-2)!!)\prod_{i=k+1}^{n} (2i-1)\right) - \\
&&\frac{2}{3}\left(\binom{n-1}{2}-1\right)((2n-1)!!) \\
&=&((2n-2)!!) \left( (2n-1) +\frac{(2n-1)(2n-3)}{3} + \frac{(n-1)^2-1}{3}\right) - \\
&&((2n-1)!!)\left( \frac{2}{3}(n-1) + \frac{2}{3}\left(\binom{n-1}{2}-1\right)\right) + \\
&&\left(\sum_{k=1}^{n-1} \frac{(5k-4)k}{3} ((2k-2)!!)\prod_{i=k+1}^{n} (2i-1)\right)  \\
&=&\frac{(5n-4)n}{3}((2n-2)!!) - \frac{2}{3}\left(\binom{n}{2}-1\right)((2n-1)!!) + \\
&&\left(\sum_{k=1}^{n-1} \frac{(5k-4)k}{3} ((2k-2)!!)\prod_{i=k+1}^{n} (2i-1)\right)\\
&=&\left(\sum_{k=1}^{n} \frac{(5k-4)k}{3} ((2k-2)!!)\prod_{i=k+1}^{n} (2i-1)\right)- \frac{2}{3}\left(\binom{n}{2}-1\right)((2n-1)!!).
\end{eqnarray*}

For part (4), suppose that 
$\sg = \sg_1 \ldots \sg_{2n} \in DU_{2n}$. Then 
only $\sg_1, \sg_2, \ldots, \sg_{2n-2}$ can match 
 $MMP(1,0,\emptyset,0)$ in $\sg$. Thus 
it cannot be the case that $\sg_{2k}=1$ where $k < n-2$ since 
then $\mmp^{(1,0,\emptyset,0)}(\sg) \leq k+1 < n-1$.  Thus 
we must have $1 \in \{\sg_{2n-4},\sg_{2n-2},\sg_{2n}\}$.  We then  
have three cases. \\

\noindent
{\bf Case 1.}  $\sg_{2n}=1$.  In this case, $\sg_{2n}$ does not match 
$MMP(1,0,\emptyset,0)$ in $\sg$.
Hence, it must be the case that 
$\mmp^{(1,0,\emptyset,0)}(\red[\sg_1 \ldots \sg_{2n-1}])=n-1$ so 
that by part (3), we have 
$D^{(1,0,\emptyset,0)}_{2n-1}(x)|_{x^{n-1}}$ ways to choose 
$\sg_1 \ldots \sg_{2n-1}$. \\

\noindent
{\bf Case 2.}  $\sg_{2n-2}=1$. In this case, $\sg_{2n-2}$ matches $MMP(1,0,\emptyset,0)$ in $\sg$. We then 
have two subcases.\\

\noindent
{\bf Subcase 2.1.} $\sg_{2n-1}= 2n$. In this case, we are guaranteed that $\sg_1$ will match 
$MMP(1,0,\emptyset,0)$ in $\sg$. Thus 
if $\tau = \red[\sg_1 \ldots \sg_{2n-3}]$, then we must have 
that 
$\mmp^{(1,0,\emptyset,0)}(\tau)+\chi(\tau_1 = 2n-3) =n-2$. We then 
have $(2n-2)$ ways to choose $\sg_{2n}$ and, once we have 
chosen $\sg_{2n}$,  
we have $\frac{2}{3}((n-2)^2-1)((2n-4)!!)$ ways to choose 
$\sg_1 \ldots \sg_{2n-3}$ by Lemma \ref{barDsectop}.  Thus the permutations 
$\sg \in DU_{2n}$ in this case contribute 
$\frac{2}{3}((n-2)^2-1)((2n-2)!!)$ to 
$C_{2n}(x)|_{x^{n-1}}$. \\

\noindent
{\bf Subcase 2.2.} $\sg_{2n-1}\neq 2n$. In this case, we are not guaranteed that $\sg_1$ will match 
$MMP(1,0,\emptyset,0)$ in $\sg$. Thus 
if $\tau = \red[\sg_1 \ldots \sg_{2n-3}]$, then we must have 
that 
$\mmp^{(1,0,\emptyset,0)}(\tau)=n-2$. We then 
have $\binom{2n-2}{2}$ ways to choose 
$\sg_{2n-1}$ and $\sg_{2n}$ and, once we have 
chosen $\sg_{2n-1}$ and $\sg_{2n}$, we have 
$D_{2n-3}(x)|_{x^{n-2}}$ ways to choose 
$\sg_1 \ldots \sg_{2n-3}$.  Thus the permutations 
$\sg \in DU_{2n}$ in this case contribute 
$\binom{2n-2}{2}D_{2n-3}(x)|_{x^{n-2}}$ to 
$C_{2n}(x)|_{x^{n-1}}$. \\

\noindent
{\bf Case 3.} $\sg_{2n-4} =1$. In this case, $\sg_{2n-4}$ matches 
 $MMP(1,0,\emptyset,0)$ in $\sg$, but 
$\sg_{2n-2}$ does not match  
 $MMP(1,0,\emptyset,0)$ in $\sg$. Again we  
have two subcases.\\

\noindent
{\bf Subcase 3.1.} $2n \in \{\sg_{2n-3},\sg_{2n-2},\sg_{2n-1},\sg_{2n}\}$. 
In this case, we are guaranteed that $\sg_1$ will match 
$MMP(1,0,\emptyset,0)$ in $\sg$. Thus 
if $\tau = \red[\sg_1 \ldots \sg_{2n-5}]$, then we must have 
that \\
$\mmp^{(1,0,\emptyset,0)}(\tau)+\chi(\tau_1 = 2n-5)=n-2$. We then 
have $\binom{2n-2}{3}$ ways to choose the remaining elements for 
$\{\sg_{2n-3},\sg_{2n-2},\sg_{2n-1},\sg_{2n}\}$.  Once we have 
chosen the remaining elements for 
$\{\sg_{2n-3},\sg_{2n-2},\sg_{2n-1},\sg_{2n}\}$,  we have 5 ways to order 
them and we have $\overline{D}_{2n-5}(x)|_{x^{n-2}}$ ways to choose 
$\sg_1 \ldots \sg_{2n-5}$.  By Lemma \ref{BarDmax}, 
$\overline{D}_{2n-5}(x)|_{x^{n-2}} =(2n-6)!!$. 
Thus the permutations 
$\sg \in DU_{2n}$ in this case contribute 
$5\binom{2n-2}{3}((2n-6)!!) = \frac{5}{6}(2n-3)((2n-2)!!)$ to 
$C_{2n}(x)|_{x^{n-1}}$. \\

\noindent
{\bf Subcase 3.2.} $2n \not \in \{\sg_{2n-3},\sg_{2n-2},\sg_{2n-1},\sg_{2n}\}$. In this case, we are not guaranteed that $\sg_1$ will match 
$MMP(1,0,\emptyset,0)$ in $\sg$. Thus 
if $\tau = \red[\sg_1 \ldots \sg_{2n-5}]$, then we must have 
that 
$\mmp^{(1,0,\emptyset,0)}(\tau)=n-2$. We then 
have $\binom{2n-2}{4}$ ways to choose the set  
$\{\sg_{2n-3},\sg_{2n-2},\sg_{2n-1},\sg_{2n}\}$.  Once we have 
chosen $\{\sg_{2n-3},\sg_{2n-2},\sg_{2n-1},\sg_{2n}\}$, 
we have 5 ways to order 
them and we have $D_{2n-5}(x)|_{x^{n-2}}$ ways to choose 
$\sg_1 \ldots \sg_{2n-5}$.  By Theorem \ref{highest}, 
$D_{2n-5}(x)|_{x^{n-2}} =(2n-6)!!-(2n-7)!!$. 
Thus the permutations 
$\sg \in DU_{2n}$ in this case contribute 
$$
5\binom{2n-2}{4}((2n-6)!!-(2n-7)!!) = \frac{5}{24}(2n-3)(2n-5)((2n-2)!!) 
- \frac{5}{3} \binom{n-1}{2} ((2n-3)!!)$$ to 
$C_{2n}(x)|_{x^{n-1}}$. \\
\ \\
It follows that 
\begin{eqnarray*}
C_{2n}(x)|_{x^{n-1}} &=& D_{2n-1}(x)|_{x^{n-1}} + \binom{2n-2}{2} D_{2n-3}(x)|_{x^{n-2}} + \\
&& \frac{2}{3}((n-2)^2-1)((2n-2)!!) + \frac{5}{6}(2n-3)((2n-2)!!)\\
&& \frac{5}{24}(2n-3)(2n-5)((2n-2)!!) - \frac{5}{3} \binom{n-1}{2} ((2n-3)!!)\\
&=& D_{2n-1}(x)|_{x^{n-1}} + \binom{2n-2}{2} D_{2n-3}(x)|_{x^{n-2}} + \\
&&\left(\frac{28n^2-72n+39}{24}\right)((2n-2)!!) - \frac{5}{3} \binom{n-1}{2} ((2n-3)!!).
\end{eqnarray*}

\end{proof}

\section{Conclusions}

As pointed out in 
\cite{kitrem2}, the simple type of recursions for 
the distribution of $\mmp^{(1,0,\emptyset,0)}(\sg)$ 
for $\sg$ in $UD_n$ or $DU_n$ proved  
in this paper no longer hold for the distribution 
of $\mmp^{(k,0,0,0)}(\sg)$ and   $\mmp^{(k,0,0,0)}(\sg)$ 
for $\sg$ in $UD_n$ or $DU_n$ if $k \geq 2$. 
For example, suppose that we try to develop a recursion 
for $A^{(2,0,\emptyset,0)}_{2n}(x) = \sum_{\sg \in UD_{2n}} 
x^{\mmp^{(2,0,\emptyset,0)}(\sg)}$. Then if we consider the permutations 
$\sg = \sg_1 \ldots \sg_{2n} \in UD_{2n}$ such that 
$\sg_{2k+1} =1$, we still have $\binom{2n-1}{2k}$ ways to 
pick the elements for $\sg_1 \ldots \sg_{2k}$.  However, 
in this case the question of whether some $\sg_i$ with 
$i \leq  2k$ matches $MMP(2,0,\emptyset,0)$ in $\sg$ 
is dependent on what values occur in $\sg_{2k+2} \ldots 
\sg_{2n}$. For example, if $2n \in \{\sg_{2k+2}, \ldots , \sg_{2n}\}$, 
then every $\sg_i$ with $i \leq k$ will match $MMP(2,0,\emptyset,0)$ in $\sg$. However, if $2n \in 
\{\sg_1, \ldots , \sg_{2k-1}\}$, this will not be the case. Thus 
we cannot develop a simple recursion for 
$A^{(2,0,\emptyset,0)}_{2n}(x)$.

However, one can develop recursions similar to 
the ones used in this paper to 
study the distribution in up-down and down-up permutations 
of other quadrant marked meshed patterns $MMP(a,b,c,d)$ in 
the case where $a,b,c,d \in \{\emptyset,1\}$. 
Indeed, in some cases, there are simple relations between such distributions 
beyond those given in Proposition \ref{prop1}.  
For example, consider the statistics 
$\mmp^{(1,0,\emptyset,0)}(\sg)$ and $\mmp^{(0,0,\emptyset,0)}(\sg)$ over 
$UD_{2n}$. Clearly, for any $\sg = \sg_1 \ldots \sg_{2n} \in UD_{2n}$, 
$\sg_{2i}$ can never match  
$MMP(1,0,\emptyset,0)$ or $MMP(0,0,\emptyset,0)$ since $(2i-1,\sg_{2i-1})$ 
will always 
be an element of $G(\sg)$ that lies in the third quadrant with respect to the coordinate 
system centered at $(2i,\sg_{2i})$.  On the other hand, elements 
of the form $\sg_{2i-1}$ for $i=1, \ldots, n$, always have 
an element $G(\sg)$ in the first quadrant relative to  the coordinate  
system centered at $(2i-1,\sg_{2i-1})$, namely, $(2i,\sg_{2i})$. 
Thus if $\sg \in UD_{2n}$, then 
$\mmp^{(1,0,\emptyset,0)}(\sg)= \mmp^{(0,0,\emptyset,0)}(\sg)$. 
Therefore, for all $n \geq 1$, 
$$A_{2n}^{(1,0,\emptyset,0)}(x) =  A_{2n}^{(0,0,\emptyset,0)}(x).$$

It is not true that $\mmp^{(1,0,\emptyset,0)}(\sg)= \mmp^{(0,0,\emptyset,0)}(\sg)$ for all $\sg \in UD_{2n+1}$ since if $\sg = \sg_1 \ldots \sg_{2n+1} 
\in UD_{2n+1}$ and $\sg_{2n+1} =1$, then $\sg_{2n+1}$ matches 
$MMP(0,0,\emptyset,0)$ in $\sg$ but does not match 
$MMP(1,0,\emptyset,0)$ in $\sg$. However, this 
is the only case where $\mmp^{(1,0,\emptyset,0)}(\sg)$ and $\mmp^{(0,0,\emptyset,0)}(\sg)$ differ. That is, if $\sg = \sg_1 \ldots \sg_{2n+1} \in UD_{2n+1}$ 
and $\sg_{2n+1} \neq 1$, then $\sg_{2n+1}$ does not match $MMP(0,0,\emptyset,0)$ in $\sg$ and we can argue as above that 
$\mmp^{(1,0,\emptyset,0)}(\sg)= \mmp^{(0,0,\emptyset,0)}(\sg)$.
However if $\sg_{2n+1} =1$, then $\sg_{2n+1}$ matches 
$MMP(0,0,\emptyset,0)$ in $\sg$ but does not match 
$MMP(1,0,\emptyset,0)$ in $\sg$.  Thus if  
$\sg \in UD_{2n+1}^{(2n+1)}$, $1+\mmp^{(1,0,\emptyset,0)}(\sg)= \mmp^{(0,0,\emptyset,0)}(\sg)$. It is easy to see that 
$$\sum_{\sg \in UD_{2n+1}^{(2n+1)}} x^{\mmp^{(1,0,\emptyset,0)}(\sg)} = 
A_{2n}^{(1,0, \emptyset,0)}(x)$$ 
so that for all $n \geq 1$, 
$$B_{2n+1}^{(0,0,\emptyset,0)}(x) + (1-x) A_{2n}^{(1,0, \emptyset,0)}(x)
= B_{2n+1}^{(1,0,\emptyset,0)}(x).$$

A slightly more subtle relation holds between 
the distribution of $\mmp^{(1,0,0,0)}(\sg)$ and $\mmp^{(1,0,\emptyset,0)}(\sg)$ for $\sg \in UD_{2n}$. For example, in \cite{kitrem2}, the authors 
computed the following table for $A_{2n}^{(1,0,0,0)}(x)$. \\
\ \\
\begin{tabular}{|l|l|}
\hline
$n$ &  $A^{(1,0,0,0)}_{2n}(x)$ \\
\hline
0 & 1 \\
\hline
1 & x\\
\hline
2 & $x^2 (3+2 x)$ \\
\hline
3 & $x^3 \left(15+30 x+16 x^2\right)$\\
\hline
4 & $x^4 \left(105+420 x+588 x^2+272 x^3\right)$\\
\hline
5 & $x^5 \left(945+6300 x+16380 x^2+18960
x^3+7936 x^4\right)$\\
\hline
6 & $x^6 \left(10395+103950 x+429660 x^2+893640 x^3+911328 x^4+
353792 x^5\right)$\\
\hline
\end{tabular} 
\ \\

Comparing the tables for  $A_{2n}^{(1,0,0,0)}(x)$ and 
$A_{2n}^{(1,0,\emptyset,0)}(x)$, one is naturally led 
to conjecture that 
for all $n \geq 1$ and $1 \leq k \leq n$, 
\begin{equation}\label{Arelation}
A_{2n}^{(1,0,\emptyset,0)}(x)|_{x^k} = A_{2n}^{(1,0,0,0)}(x)|_{x^{2n-k}}.
\end{equation}
This follows from comparing 
$\sg = \sg_1 \ldots \sg_{2n} \in UD_{2n}$ with its reverse complement 
$(\sg^r)^c = (2n+1 -\sg_{2n}) (2n+1 -\sg_{2n-1}) \ldots (2n+1-\sg_1)$ which 
is also in $UD_{2n}$.  That is, suppose that 
$\sg_i$ matches $MMP(1,0,\emptyset,0)$ in $\sg$. Then 
$i$ must be odd, i.e. $i = 2k+1$ for some $0 \leq k \leq n-1$, and there 
must be no elements 
in $\sg_1 \ldots \sg_{2k}$ which are less than $\sg_{2k+1}$. This 
means that in $(\sg^r)^c$, $(2n+1 -\sg_{2k+1})$ has no elements 
to its right which are greater than $(2n+1 -\sg_{2k+1})$ so that 
 $(2n+1 -\sg_{2k+1})$ will not match $MMP(1,0,0,0)$ in $(\sg^r)^c$.
Vice versa, if $\sg_{2k+1}$ does not match $MMP(1,0,\emptyset,0)$ 
in $\sg$, there is an element 
in $\sg_1 \ldots \sg_{2k}$ which is less than $\sg_{2k+1}$. This 
means that in $(\sg^r)^c$, $(2n+1 -\sg_{2k+1})$ has an element  
to its right which is  greater than $(2n+1 -\sg_{2k+1})$ 
so that $(2n+1 -\sg_{2k+1})$ will match 
$MMP(1,0,0,0)$ in $(\sg^r)^c$.  Similarly, in $\sg$, none 
of $\sg_2, \sg_4, \ldots , \sg_{2n}$ will match $MMP(1,0,\emptyset,0)$ 
while in $(\sg^r)^c$, each of 
$(2n+1 -\sg_2), \ldots, (2n+1-\sg_{2n})$ will match 
$MMP(1,0,0,0)$ in $(\sg^r)^c$.  Thus it follows 
that for all $\sg \in UD_{2n}$, 
$$n+(n-\mmp^{(1,0,\emptyset,0)}(\sg)) = \mmp^{(1,0,0,0)}((\sg^r)^c).$$
This shows that (\ref{Arelation}) holds.

There is no such simple relation between the distribution 
of $\mmp^{(1,0,0,0)}(\sg)$ and the distribution of 
$\mmp^{(1,0,\emptyset,0)}(\sg)$ for $UD_{2n+1}$, $DU_{2n}$ or 
$DU_{2n+1}$ as can be seen from the following tables 
computed in \cite{kitrem2}. \\
\ \\
\begin{tabular}{|l|l|}
\hline
$n$ &  $B^{(1,0,0,0)}_{2n-1}(x)$ \\
\hline 
1& 1 \\
\hline
2 & $2 x$\\
\hline
3 & $8 x^2 (1+x)$ \\
\hline 
4 & $16 x^3 \left(3+8 x+6 x^2\right)$ \\
\hline
5& $128 x^4 \left(3+15 x+27 x^2+17 x^3\right)$ \\
\hline 
6 & $256 x^5 \left(15+120 x+381
x^2+556 x^3+310 x^4\right)$\\
\hline
7 & $1024 x^6 \left(45+525 x+2562 x^2+6420 x^3+8146 x^4+4146 x^5\right)$\\
\hline
\end{tabular}\\
\ \\
\ \\
\begin{tabular}{|l|l|}
\hline
$n$ &  $C^{(1,0,0,0)}_{2n}(x)$ \\
\hline
0 & 1 \\
\hline
1 & 1 \\
\hline 
2 & $x (2+3 x)$\\
\hline 
3 & $x^2 \left(8+28 x+25 x^2\right)$ \\
\hline 
4 & $x^3 \left(48+296 x+614 x^2+427 x^3\right)$ \\
\hline 
5 & $x^4 \left(384+3648 x+13104 x^2+20920
x^3+12465 x^4\right)$\\
\hline 
6 & $x^5 \left(3840+51840 x+282336 x^2+769072 x^3+
1039946 x^4+555731 x^5\right)$
\\
\hline 
\end{tabular}\\
\ \\
\ \\
\begin{tabular}{|l|l|}
\hline
$n$ &  $D^{(1,0,0,0)}_{2n-1}(x)$ \\
\hline
1 & 1 \\
\hline 
2 & $x (1+x)$ \\
\hline 
3 & $x^2 \left(3+8 x+5 x^2\right)$ \\
\hline 
4 & $x^3 \left(15+75 x+121 x^2+61 x^3\right)$ \\
\hline 
5 & $x^4 \left(105+840 x+2478 x^2+3128 x^3+1385
x^4\right)$ \\
\hline 
6 & $x^5 \left(945+11025 x+51030 x^2+115350 x^3+124921 x^4+50521 x^5\right)$ \\
\hline 
7 & $x^6 \left(10395+166320 x+1105335 x^2+3859680 x^3+7365633 x^4+7158128
x^5+2702765 x^6\right)$\\
\hline
\end{tabular}
\ \\

Based on these tables, we conjectured in 
\cite{kitrem2} that the polynomials $A^{(1,0,0,0)}_{2n}(x)$, $B^{(1,0,0,0)}_{2n+1}(x)$, $C^{(1,0,0,0)}_{2n}(x)$, and 
$D^{(1,0,0,0)}_{2n+1}(x)$ are unimodal for all $n \geq 1$.  We also 
conjecture that  $A^{(1,0,\emptyset,0)}_{2n}(x)$, $B^{(1,0,\emptyset,0)}_{2n+1}(x)$, $C^{(1,0,\emptyset,0)}_{2n}(x)$, and 
$D^{(1,0,\emptyset,0)}_{2n+1}(x)$ are unimodal for all $n \geq 1$.

Finally, we suggest that it should be interesting to study the distribution of 
quadrant marked mesh patterns on other classes of pattern-restricted permutations such as 2-{\em stack-sortable permutations} or {\em vexillary permutations}
(see \cite{kit} for definitions of these) and many other permutation classes having nice properties.

\end{document}